\documentclass{amsart}[11pt,a4paper]
\usepackage[latin1]{inputenc}
\usepackage[english]{babel}
\usepackage{geometry}
\usepackage{array}
\usepackage{enumerate}
\usepackage{amsmath, amsfonts, amscd, amssymb, amsthm, xypic, mathtools, stmaryrd}
\usepackage{epsfig}

\usepackage[usenames,dvipsnames]{color}

\usepackage[backref=page]{hyperref}

\hypersetup{
 colorlinks,
 citecolor=green,
 linkcolor=blue,
 urlcolor=Blue}

\newcommand{\D}{\mathcal{D}}
\renewcommand{\o}{\omega}

\newcommand{\C}{\mathbb{C}}
\renewcommand{\c}{\mathcal{C}}

\newcommand{\N}{\mathcal{N}}
\newcommand{\n}{\mathfrak{n}}
\newcommand{\m}{\mathfrak{m}}

\newcommand{\I}{\mathcal{I}}
\newcommand{\sT}{L^{-1}}
\newcommand{\sK}{L}
\newcommand{\F}{\mathcal{F}}
\renewcommand{\L}{\mathcal{L}}

\newcommand{\cC}{\mathcal{C}}
\renewcommand{\cD}{\mathcal{D}}
\newcommand{\un}{\underline}
\newcommand{\wt}{\widetilde}
\renewcommand{\F}{\mathcal{F}}

\newcommand{\M}{\mathcal{M}}
\newcommand{\sM}{\mathfrak{M}}

\newcommand{\A}{\mathcal{A}}

\newcommand{\Z}{\mathbb{Z}}

\newcommand{\ov}{\overline}

\renewcommand{\S}{\mathcal{S}}
\renewcommand{\O}{\mathcal{O}}

\DeclareMathOperator{\Aut}{Aut}
\DeclareMathOperator{\Ber}{Ber}
\DeclareMathOperator{\Spec}{Spec}
\DeclareMathOperator{\Sym}{Sym}
\DeclareMathOperator{\Hom}{Hom}
\DeclareMathOperator{\End}{End}
\DeclareMathOperator{\rk}{rk}
\DeclareMathOperator{\red}{red}
\DeclareMathOperator{\spli}{split}
\DeclareMathOperator{\ev}{ev}

\DeclareMathOperator{\Isom}{Isom}
\DeclareMathOperator{\bos}{bos}
\DeclareMathOperator{\id}{id}
\DeclareMathOperator{\loc}{loc}

\DeclareMathOperator{\FIB}{FIB}

\DeclareMathOperator{\Gr}{GR}
\DeclareMathOperator{\eGr}{GR^{et}}
\DeclareMathOperator{\sGr}{sGR}
\DeclareMathOperator{\esGr}{sGR^{et}}

\DeclareMathOperator{\DM}{ST^{DM}}
\DeclareMathOperator{\DMs}{sST^{DM}}

\DeclareMathOperator{\St}{ST}
\DeclareMathOperator{\sSt}{sST}
\DeclareMathOperator{\eSt}{ST^{DM}}
\DeclareMathOperator{\esSt}{sST^{DM}}

\usepackage{calligra,mathrsfs}
\DeclareMathOperator{\HHom}{\mathscr{H}\text{\kern -3pt {\calligra\large om}}\,}
\DeclareMathOperator{\EExt}{\mathscr{E}\text{\kern -3pt {\calligra\large xt}}\,}

\newenvironment{sis}{\left\{\begin{aligned}}{\end{aligned}\right.}

\newtheorem{theorem}{Theorem}[section]
\newtheorem{proposition}[theorem]{Proposition}
\newtheorem{fact}[theorem]{Fact}
\newtheorem{lemma}[theorem]{Lemma}
\newtheorem{corollary}[theorem]{Corollary}
\newtheorem{definition}[theorem]{Definition}

\newtheorem{theoremalpha}{Theorem}

\theoremstyle{definition}

\newtheorem{remark}[theorem]{Remark}

\title{Moduli and periods of Supersymmetric Curves}
\date{}
\thanks{Both authors are funded by the FIRB 2012 Moduli Spaces and their Applications, and they acknowledge the support of the Simon center for their participation to the Supermoduli workshop in Stony Brook. GC also acknowledges the support by the ``De Giorgi Center" for his participation to the intensive research period ``Perspectives in Lie Theory", where part of this work was carried out.}

\author[G. Codogni]{Giulio Codogni}
\address{Dipartimento di Matematica e Fisica, Universit\`{a} Roma Tre, Largo San Leonardo Murialdo, 1
00146, Roma, Italy.}
\email{codogni@mat.uniroma3.it}
\author[F. Viviani]{Filippo Viviani}
\address{Dipartimento di Matematica e Fisica, Universit\`{a} Roma Tre, Largo San Leonardo Murialdo, 1
00146, Roma, Italy.}
\email{viviani@mat.uniroma3.it}

\begin{document}

\maketitle

\begin{abstract}Supersymmetric curves are the analogue of Riemann surfaces in super geometry. We establish some foundational results about (Deligne-Mumford) complex superstacks, and we then prove that the moduli superstack of supersymmetric curves is a smooth Deligne-Mumford complex superstack. We then show that the superstack of supersymmetric curves  admits a coarse complex superspace, which, in this case, is just an ordinary complex space. In the second part of this paper we discuss the period map. We remark that the period domain is the moduli space of ordinary abelian varieties endowed with a symmetric theta divisor, and we then show that the differential of the period map is surjective. In other words, we prove that any first order deformation of a classical Jacobian is the Jacobian of a supersymmetric curve.
\end{abstract}

\begin{section}{Introduction}

Supergeometry is one of the first example of non-commutative geometry. In this field, one considers spaces of some type (differentiable, complex, algebraic, etc...) with both commuting and anti-commuting coordinates. Supersymmetric (or susy for simplicity) curves are the generalization in this setting of Riemann surfaces. The main motivations for studying susy curves and their moduli superspaces come from supersymmetric string theory, where they play a central role. 

In the first part of this paper, after  some foundational material about complex superstacks, we construct the moduli complex superstack of susy curves and we show that it admits a coarse superspace, which turns out to be an ordinary complex space.
 In the second part of the paper, we give an algebraic description of the period map, and we prove that the differential of the period map (whenever defined) is surjective.

Our motivation is twofold. First, we hope to help the development of supersymmetric string theory. Secondly, we would like to point out a strong connection between susy curves and classical moduli spaces, aiming to give a new insight into these objects. Let us briefly go trough the results and the organization of this paper.

In Section \ref{sec:prel}, we provide a self contained introduction to complex supergeometry. We want to establish the basic definitions and notation in a way that it is convenient for superstacks. The main novelty is the description of the canonical involution $\Gamma$ of a (complex) superspace $X$, and the corresponding quotient $X/\Gamma$ by $\Gamma$, which we call the \emph{bosonic quotient}. The bosonic quotient functor is left adjoint to the natural inclusion $i$ of (complex) spaces into (complex) superspaces; while its right adjoint is the more familiar bosonic truncation $X\mapsto X_{\bos}$. 
Moreover, we introduce smooth and \'etale morphisms and we investigate their behaviors with respect to the above three functors (natural inclusion, bosonic truncation and bosonic quotient), see Propositions \ref{P:etale} and \ref{P:smooth}.

Taking a rather algebraic point of view, our main references are \cite[Chap. 4]{Manin1} and \cite{Vaintrob}. Other references about supergeometry are \cite{Deligne}, \cite{Fioresi}, \cite{Witten1}, \cite{ProjEmbed} and the first section of \cite{DW1}.


Section \ref{sec:superstack} is devoted to introduce complex (resp. Deligne-Mumford complex) superstacks and to prove some foundational results about them.

In subsection \ref{S:CFG}, we introduce the $2$-category $\FIB_{\S}$ of categories fibered in groupoids (=CFG) over the category $\S$ of complex superspaces and the $2$-category $\FIB_{\S_{\rm ev}}$ of CFG over the category $\S_{\rm ev}$ of complex spaces. The $2$-Yoneda lemma provides a fully faithful embedding of $\S$ into $\FIB_{\S}$ and of $\S_{\rm ev}$ into $\FIB_{\S_{\rm ev}}$. We define a pair of adjoint $2$-functors $i:\FIB_{\S_{\rm ev}}\to \FIB$ and $(-)_{\bos}:\FIB_{\S}\to \FIB_{\S_{\rm ev}}$, which extend the functors $i:\S_{\rm ev}\to \S$ and $(-)_{\bos}:\S\to \S_{\rm ev}$. 


In subsection  \ref{S:stacks}, we introduce superstacks (resp. stacks) as a CFG over $\S$ (resp. $\S_{\rm ev}$) that satisfy descent theory with respect to the \'etale topology. We prove that the $2$-functors $i$ and $(-)_{\bos}$ restrict to the $2$-subcategories of stacks and superstacks (see Lemma \ref{L:i-bos-stacks}), and that the CFG over $\S$ associated to a complex superspace  is a superstack and similarly for complex spaces (see Lemma \ref{L:repres}). 

In subsection \ref{S:algstacks}, we introduce the $2$-category $\sSt$ (resp. $\esSt$) of complex (resp. Deligne-Mumford=(DM) complex) superstacks as superstacks having a separated and representable diagonal and possessing a smooth (resp. \'etale) atlas, and similarly the $2$-category $\St$ (resp. $\eSt$) of complex (resp. DM complex) stacks. 
It is then clear that the CFG over $\S$ associated to a complex superspace is a DM complex superstack, and similarly for complex spaces. 
In Proposition \ref{P:i-bos-DM}, we prove that  the two $2$-functors $i$ and $(-)_{\bos}$ restrict to $2$-functors between the $2$-subcategories $\St^{({\rm et})}$ and $\sSt^{({\rm et})}$. 


In subsection \ref{SS:supergrp}, we introduce  the $2$-category $\sGr$ (resp. $\esGr$) of complex (resp. \'etale complex) supergroupoids $X_1\rightrightarrows X_0$, and similarly the $2$-category $\Gr$ (resp. $\eGr$) of complex (resp. \'etale complex) groupoids. 
The relation between (resp. \'etale) complex (super)groupoids and (resp. DM) complex (super)groupoids comes from the existence of the  essentially surjective realization $2$-functors:
\begin{equation*}\label{E:Fintro}
\begin{aligned}
\F: \sGr^{({\rm et})} &\longrightarrow \sSt^{({\rm DM})}\\
X_1\rightrightarrows X_0 & \mapsto [X_1\rightrightarrows X_0],
\end{aligned}
\hspace{0.5cm} 
\text{ and }
\hspace{0.5cm}
\begin{aligned}
\F_{\ev}: \Gr^{({\rm et})} &\longrightarrow \St^{({\rm DM})}\\
X_1\rightrightarrows X_0 & \mapsto [X_1\rightrightarrows X_0].
\end{aligned}
\end{equation*} 
In Proposition \ref{P:i-bosDM}, we define the natural inclusion $i:\Gr^{({\rm et})}\to \sGr^{({\rm et})}$ and the bosonic truncation $(-)_{\bos}:\sGr^{({\rm et})}\to \Gr^{({\rm et})}$ and we prove that they commute with the realization functors $\F$ and $\F_{\ev}$. 
Furthermore, adapting the results of \cite{Moe} and \cite{Pro} to our setup, we deduce that the restrictions of the realization $2$-functors $\F$ and $\F_{\ev}$ from \'etale complex (super)groupoids to DM complex (super)groupoids induce isomorphisms of bicategories 
$\esGr[W^{-1}] \cong \DMs$ and  $  \sGr[W_{\ev}^{-1}]\cong \DM$, where $\esGr[W^{-1}]$ (resp. $\sGr[W_{\ev}^{-1}]$) is the localization of $\esGr$ (resp. $\eGr$) with respect to the collection $W$ (resp. $W_{\ev}$) of weak equivalences of \'etale complex (super)groupoids.
Finally, in Theorem \ref{T:DMtrunc}, we define the bosonic quotient  $2$-functor $-/\Gamma: \esGr\to \eGr$ and we prove that it descend to the localization with respect to weak equivalences of \'etale complex (super)groupoids, so that we get an induced pseudofunctor $-/\Gamma:\esSt\to \eSt$
which is left adjoint to the natural inclusion functor $i$. 

We can summarize all the results proved in Section \ref{sec:superstack} into the following 


\begin{theoremalpha}[Proposition \ref{P:i-bos-DM}, Proposition \ref{P:i-bosDM} and Theorem \ref{T:DMtrunc} ]\label{T:A}
\noindent 
\begin{enumerate}
\item \label{T:A1} We have the following commutative diagram of $2$-functors among $2$-categories
 $$\xymatrixcolsep{5pc}  \xymatrix{ 
\sGr \ar@/_1pc/[r]_{(-)_{\bos}} \ar[d]_{\F}& \Gr  \ar@{_{(}->}[l]^{i} \ar[d]^{\F_{\ev}} \\
\sSt  \ar@/_1pc/[r]_{(-)_{\bos}}  & \St  \ar@{_{(}->}[l]^{i} \\ 
}$$
and moreover $(-)_{\bos}$ is right adjoint to $i$. 
\item \label{T:A2} We have the following commutative diagram of pseudofunctors among bicategories
 $$\xymatrixcolsep{5pc}  \xymatrix{ 
 \esGr \ar@/_1pc/[r]_{(-)_{\bos}} \ar[dd]_{\F} \ar@/^1pc/[r]^{-/\Gamma} & \eGr \ar@{_{(}->}[l]^{i} \ar[dd]^{\F_{\ev}}   \\
  &  \\
\esGr[W^{-1}] \cong \esSt  \ar@/_1pc/[r]_{(-)_{\bos}}  \ar@/^1pc/[r]^{-/\Gamma}  &    \eSt  \cong \eGr[W_{\ev}^{-1}]  \ar@{_{(}->}[l]^{i} \\ 
}$$
and moreover $(-)_{\bos}$ is right adjoint to $i$ and $-/\Gamma$ is left adjoint to $i$. 
\end{enumerate}
\end{theoremalpha}





We are not able to define the bosonic quotient  $2$-functor (or pseudofunctor) from (general) complex superstacks to (general) complex stack. There are several reasons which prevented us in defining such a functor. First of all, we do not know how to extend the bosonic quotient functor $/ \Gamma:\S \to \S_{\rm ev}$ to CFG's,  the reason being that $/\Gamma$ does not admit a left adjoint functor (since it does not preserve fibre products), see Remark \ref{R:NObosquot}.  If we, instead, try to define the bosonic truncation using atlas (and therefore passing through complex supergroupoids), we are stacked with two further problems. First of all, we do not know if the bicategory of complex (super)stacks is the localization of the bicategory of complex (super)groupoids with respect to weak equivalences since the arguments of \cite{Moe} and \cite{Pro} seem only to work for \'etale (super)groupoids and DM (super)stacks.  And secondly, the bosonic quotient of a complex supergroupoids does not seem to be  a complex groupoids (and not even a groupoid at all!) due to the fact that the bosonic quotient does not preserve fiber products and smoothness (see Remark \ref{R:bad-Gamma}).  We think that these problems deserve further investigation.


Note that our results can easily be extended from (DM) complex superstacks to (DM) algebraic superstacks. We have chosen  to work in the category of complex superspaces/superstacks since the construction of the superstack of susy curves (see the discussion that follows) is easier in the complex category (due to the existence of Kuranishi families). For an introduction to  algebraic superstack, see also   \cite{AG}.


In Section \ref{sec:susy}, we introduce the notion of a susy curve over a base complex superspace, and we prove all the results which are needed for the construction of the moduli space: the equivalence of susy curves over an ordinary complex space with spin curves (see subsection \ref{S:susy-spin}), the representability of the isomorphism functor between two susy curves over the same base (see subsection \ref{S:automorphism}),  the construction of the Kuranishi families for a susy curve over a point (see subsection \ref{S:Kuranishi}). 
References about supersymmetric curves are \cite{Manin2}, \cite{Ber} and \cite{Witten2}; other sources are  \cite{Giulio}, \cite{LeBrun}, \cite{Rabin}, \cite{TR}, \cite{FK}, \cite{FKAut} and \cite{K}. We only consider for simplicity susy curves of genus $g\geq 2$; susy curves of genus $0$  and $1$ are studied in \cite{FR}, \cite[Sections 2.7-8]{Manin2}, \cite{REL}. Other topics on susy curves that we do not discuss are: punctures of Neveu-Schwarz and Ramond type  (see \cite{Witten2}) and theta functions (see \cite{Tsu}).

In Section \ref{sec:moduli}, we introduce and study the moduli superstack $\sM_g$ of susy curves of genus at least $2$. We summarize the main results that we get into the following

 
\begin{theoremalpha}[= Theorem \ref{T:sS-superstack} and Corollary \ref{cor:coarse}]\label{T:B}
Let $g\geq 2$. 
\noindent 
\begin{enumerate}[(1)]
\item \label{T:B1} $\sM_g$ is a smooth and separated DM complex superstack of dimension $3g-3|2g-2$ whose bosonic truncation $(\sM_g)_{\bos}$ is the complex stack $\S_g$ of spin curves of genus $g$.  Moreover, $\sM_g$ has two connected components, denoted by $\sM_g^+$ and $\sM_g^-$, whose bosonic truncations are $(\sM_g)_{\bos}^+$ and $(\sM_g)_{\bos}^-$ are the complex stacks $\S_g^+$ and $\S_g^-$ of, respectively, even and odd spin curves of genus $g$.  
\item \label{T:B2} There exists a  coarse moduli superspace $\mathbb{M}_g$ for $\sM_g$, which is indeed an ordinary complex space and it is also the coarse moduli space for the bosonic quotient  $\sM_g/\Gamma$. 
The  complex space  $\mathbb{M}_g$ is  non-reduced; its underlying reduced complex space $(\mathbb{M}_g)_{\red}$ is isomorphic to the coarse moduli space $S_g$ of  spin curves of genus $g$. In particular, $\mathbb{M}_g$ is separated, and it has two connected components whose 
underlying reduced spaces are the coarse moduli spaces $S_g^+$ and $S_g^-$ of, respectively, even and odd spin curves of genus $g$. 

\end{enumerate}
\end{theoremalpha}


We prove part \eqref{T:B1} of the above theorem (which is a folklore result  among the experts) by gluing together (a finite number of) Kuranishi families of susy curves, miming the construction of the moduli stack of Riemann surfaces given by Arbarello-Cornalba in \cite{AC} (see also  \cite[Chap. XII.4]{ACG}). 
Regarding the existence of a coarse moduli space as in \eqref{T:B2}, note that, in contrast with the classical case, it is not clear to us if every separated DM complex   superstack has a coarse complex superspace. We establish part \eqref{T:B2} of the above theorem by showing that  any map from $\sM_g$ onto a complex superspace has to factor through the bosonic quotient $\sM_g/\Gamma$ (see Proposition \ref{prop:fact}), a property which uses in a crucial way  the existence of the canonical automorphism on  every susy curve.  

The above theorem leaves open some natural questions. For example, it would be interesting to know whether $\mathbb{M}_g$ is quasi-projective and if  it can be interpreted as a coarse moduli space for some classical objects (e.g. spin curves endowed with some extra structure). Also it would be very interesting to investigate whether $\sM_g$ and $\mathbb{M}_g$ could be constructed as quotients of Hilbert superschemes of pluri-canonically embedded susy curves, as in the classical case (see \cite[Section 3]{Giulio} for a discussion of this approach). This could also lead to the construction of $\sM_g$ and $\mathbb{M}_g$ as algebraic superstacks/superspaces: our choice of working with complex superstacks/superspaces is due to the fact that Kuranishi families are easier to construct in complex supergeometry, and the Hilbert superscheme approach could bypass this problem (see \cite[Chap. XII.5]{ACG} for a construction of the algebraic stack of curves using Hilbert schemes of pluri-canonically embedded curves).  And finally, it would be interesting to write down the details of the compactification of $\sM_g$ via stable susy curves proposed  by P. Deligne in a letter to Y. Manin back in 1987 (recently this letter has been posted on Deligne's webpage \cite{DeligneLetter}), see also \cite{BaS} and \cite{Coh}. 


There are in the literature other approaches to the construction of the moduli superspace/superstack of susy curves, which we now briefly review. First of all, fine moduli superspaces of susy curves with level $n$-structures (for $n\geq 3$) have been constructed  as complex ``canonical superorbifolds'' by LeBrun and Rothstein in \cite{LeBrun} and as complex algebraic superspaces by Dom\'inguez-Perez, Hern\'andez-Ruip\'erez and  Sancho de Salas in \cite{spain}.
The advantage of our approach using complex superstacks is that we can work with moduli of susy curves without any level $n$-structure. 
A completely different approach to the moduli of susy curves consists in using the uniformization theorem for susy curves and  the super Teichm\"uller theory, see \cite{Rab1}, \cite{CR}, \cite{BB}, \cite{Bry}, \cite{Hod1}, \cite{Hod2}, \cite{Hod3}, \cite{Hod4}, \cite{UY1}, \cite{UY2}, \cite{UY3}, \cite{BH}.
Recently, Donagi and Witten have studied in \cite{DW2} and \cite{DW1} the global geometry of the moduli superstack of susy curves (assuming for granted its existence)  by showing that it is not  projected (hence in particular not split) for genera $g\geq 5$, i.e. the natural inclusion of $(\sM_g)_{\bos}$ in $\sM_g$ does not admit a section; a problem which was originally investigated (with some partial results)   by Falqui and Reina in \cite{Fal}, \cite{FR1}, \cite{FR2}.



In section \ref{sec:period}, we study periods of susy curves. 
We introduce the period map as a rational morphism
$$
P\colon \sM_g^+\dashrightarrow \N_g 
$$
where $\N_g$ is the classical moduli stack of $g$ dimensional abelian varieties endowed with a symmetric theta divisor. The map $P$ is defined on the open subset of split susy curves $C_L\in \sM_g^+$ such that $h^0(L)=0$. Its bosonic truncation $P_{\rm bos}:\S_g^+\dashrightarrow \N_g$ is the map that sends a spin curve $(C,L, \phi)$ into the the Jacobian $J(C)$ of $C$ endowed with the symmetric theta divisor associated to the spin structure. Note that $P_{\rm bos}$ is defined on the entire moduli stack $\S_g$.

The period map $P$ factors through the quotient $\sM_g^+/\Gamma$ and it gives rise to a rational map 
$$
P/\Gamma\colon \sM_g^+/\Gamma\dashrightarrow \N_g.
$$
In section \ref{sec:period}, we focus on the infinitesimal period map, i.e. the differential $d(P/\Gamma)$ of the period map $P/\Gamma$. Our results are summarized into the  following  


\begin{theoremalpha}[= Theorem \ref{diff_per} and Theorem \ref{torelli}]
Let $C_L\in \sM_g^+/\Gamma$ be a split susy curve such that $h^0(L)=0$.  
\begin{enumerate}
\item The infinitesimal period map at $C_L$
$$ 
d(P/\Gamma)_{C_L}\colon T_{[C_L]}(\sM_g^+/\Gamma)=H^2( C\times C, \sT\boxtimes \sT(-\Delta))^+\to T_{P(C_L)}\N_g =\Sym^2H^1(C,\O)
$$
is the even part of the $H^2$ of the morphism of sheaves on $C\times C$
$$
\sT\boxtimes \sT(-\Delta) \hookrightarrow \O
$$
defined by the multiplication with the Szeg\H{o} kernel $S_L$ associated to $L$.

\item The infinitesimal period map $d(P/\Gamma)_{C_L}$ at $C_L$ is surjective. 
\end{enumerate}
 \end{theoremalpha} 


 Part (1) of the above Theorem follows from an analytic formula (that we review in subsection \ref{sec:class_formula}) for the period map which was first discovered by   D'Hoker and Phong in \cite{DP} and then  improved and expanded by Witten in \cite[Section 8.3]{Witten2}. We prove part (1)  in subsection \ref{sec:inf_per}, after a description of the tangent spaces $T_{[C_L]}(\sM_g^+/\Gamma)$ and $T_{P(C_L)}\N_g$,  and a review of the construction of the Szeg\H{o} kernel $S_L$ associated to theta characteristic $L$ following the algebraic approach of Ben-Zvi and Biswas in \cite{BZB}. 
Our description of the infinitesimal period map as being induced by a morphism of sheaves on $C\times C$ is influenced by a paper of Bruno and Sernesi \cite{BS}. 
Part (2) of the above Theorem says, in a more fancy language, that any first order deformation of a generic classical Jacobian is the Jacobian of a susy curve. This is in contrast with the classical case: if the genus is at least $4$, the generic first order deformation of a classical Jacobian is not a Jacobian anymore.  Let us also point out that, in contrast to the classical case, the infinitesimal period map for susy curves cannot be injective for dimensional reasons, see Remark \ref{R:not-inj}.


The relation between periods and the superstring measure is studied in the recent preprint \cite{FKP}. Periods could also be defined for susy curves with Ramond punctures, cf. \cite{WittenPunctures}; in this more general case, which we do not discuss, the period domain is no longer an ordinary space, but a super analogue of the Siegel upper half space. A reference about the super Siegel upper half space is \cite{FioresiSiegel}.

\begin{subsection}*{Acknowledgements}
This project started during the Supermoduli workshop, which took place in Stony Brook on May 2015. We would like to thank the organizers and the lecturers for the stimulating environment. We had the pleasure and the benefit of conversation about the topics of this paper with A. Bruno, T. Covolo, R. Donagi, S. Kwok, R. Fioresi and E. Witten.
We thank U. Bruzzo, D. Hern\'{a}ndez Ruip\'{e}rez and A. Polishchuk for useful comments about the first version of this paper.
\end{subsection}

\end{section}

\begin{section}{Complex superspaces}\label{sec:prel}

The aim of this section is to recall the definition and the basic properties of complex superspaces, and prove a few facts needed for the theory of super stacks. A \textbf{superspace} is a locally superringed space $(X,\O_X)$, i.e. $X$ is a topological space (called the underlying topological space) and $\O_X=\O_{X,0}\oplus \O_{X,1}$ is a sheaf of supercommutative $\mathbb{C}$-superalgebras (called the structure sheaf) such that the stalk $\O_{X,x}$ at any point $x\in X$ is a local super ring, i.e. it has a unique maximal homogenous ideal $\m_x$. When no confusion seems to arise, we will denote the superspace $(X,\O_X)$ simply by $X$. 

A morphism of superspaces from $(X,\O_X)$ to $(Y,\O_Y)$ is a pair $(f,f^{\sharp})$ consisting of a continuous map $f:X\to Y$ and a map $f^{\sharp}:\O_Y\to f_*\O_X$ of  sheaves of  $\C$-superalgebras (thus preserving the parity) such that for any point $x\in X$ the induced homomorphism $f^{\sharp}_x:\O_{Y,f(x)}\to \O_{X,x}$ is local, i.e. $f^{\sharp}_x(\m_{f(x)})\subset \m_x$. When no confusion seems to arise, we will denote such a morphism simply by $f$. 

Coordinates at a point $x$ of superspace $X$ are homogeneous elements of $\m_x\setminus \m_x^2$ which generate the ideal $\m_x$. Given a morphism $f\colon X\to Y$, vertical or relative coordinates at a point $x$ in $X$ are homogeneous elements of $\m_x\setminus \m_x^2$ whose image generate $\m_x/f^{\sharp}_x(\m_{f(x)})$.

An (ordinary) \emph{space}, i.e. a locally ringed space $(M,\O_M)$ where $M$ is a topological space and $\O_M$ is a sheaf of commutative $\C$-algebras such that $\O_{M,x}$ is a local ring for every $x\in M$, is in a natural way a superspace by regarding  $\O_M$ as a sheaf of supercommutative $\C$-superalgebras concentrated in degree $0$. In this way, we get a fully faithful embedding $i$ of the category of  spaces into the category of superspaces whose essential image consists of all the superspaces $X$ such that $\O_{X,1}=0$  (such superspace will be called \emph{even} or \emph{bosonic}). We are going to show that the inclusion $i$ has a left and a right adjoint; in particular, this implies that $i$ respects the fibers products (which are defined via the usual tensor product of sheaves in each of  the two categories).

Given a superspace $X$, we will denote by $\n_X$ the ideal sheaf in $\O_X$ generated by the elements of odd degree, i.e. $\n_X=\O_{X,1}^2\oplus \O_{X,1}\subset \O_{X,0}\oplus \O_{X,1}$.  
To any superspace $X$, we can associate the bosonic superspace (i.e. the space)  $X_{\bos}=(X,\O_X/\n_X)$, which  is called the \emph{bosonic truncation} of $X$. There is a natural  closed embedding of $X_{\bos}\hookrightarrow X$ which is an identity on the underlying topological spaces and it is the quotient map $\O_X\twoheadrightarrow \O_X/\n_X$ on the sheaf of functions. The functor $X\mapsto X_{\bos}$ is a right adjoint of the inclusion $i$. In particular, the inclusion $X_{\bos}\hookrightarrow X$ is universal with respect to morphism from spaces into $X$. Moreover,  the bosonic truncation respects  fibre products since it has a left adjoint; in symbols, we have that $(X\times_Y Z)_{\bos}=X_{\bos}\times_{Y_{\bos}}Z_{\bos}$.
More generally, for any $n\geq 1$, we can consider the closed superspace $X^{(n)}:=(X,\O_X/\n_X^n)$, which is called the \emph{$n$-th infinitesimal neighborhood} of $X_{\bos}$ in $X$.

Any superspace $X$ admits a \emph{canonical automorphism} $\Gamma_X$ (or simply $\Gamma$ when no confusion arises) which is the identity on the topological space $X$ and on the even degree part $\O_{X,0}$ of the structure sheaf $\O_X$, while it acts as $-1$ on the odd degree part $\O_{X,1}$ of the structure sheaf $\O_X$. Remark that any morphism $f\colon X\to Y$ is $\Gamma$-equivariant, i.e. $\Gamma_Y\circ f=f\circ \Gamma_X$. Given a superspace $X=(X,\O_X)$, the quotient $X/\Gamma$ (which we call the \emph{bosonic quotient}) by the canonical automorphism is the superspace $(X,\O_X^\Gamma)$, where $\O_X^{\Gamma}$ is the subsheaf of $\O_X$ consisting of invariants elements. By the definition of $\Gamma$, $\O_X^{\Gamma}$ is a sheaf of supercommutative $\C$-superalgebras concentrated in degree $0$, so that $X/\Gamma$ is a bosonic superspace (i.e. a space). Note that there is a natural surjective map $X\twoheadrightarrow X/\Gamma$ which is the identity on the underlying topological spaces and it is the inclusion $\O_X^{\Gamma}\hookrightarrow \O_X$ at the level of structure sheaves. The functor $X\mapsto X/\Gamma$ is a left adjoint of the inclusion $i$. In particular, the morphism $X\twoheadrightarrow X/\Gamma$ is universal with respect to morphisms from $X$ into  spaces. 

Note that the composition 
$$X_{\bos} \hookrightarrow X \twoheadrightarrow X/\Gamma, $$
is a morphism of spaces such that $(X_{\bos})_{\red}=(X/\Gamma)_{\red}$, where for a space $Z$ we denote by $Z_{\rm red}$ the space whose underlying topological space is the same as the one of $Z$ and whose structure sheaf  $\O_{Z_{\red}}$ is obtained from $\O_Z$ by quotienting out the nilpotent elements.


We will (most of the times) be working with complex superspaces, which are superspaces  locally modeled on  closed subspaces of complex superdomains, as we are now going to define.
For any $p,q,\geq 0$, the superspace $\C^{p|q}$ is $(\C^p, \O_{\C^p}\otimes \bigwedge^* V$), where $\O_{\C^p}$ is the sheaf of holomorphic functions on $\C^p$, $V$ is a $q$ dimensional complex vector space, and we give degree zero to the element of $\O_{\C^p}$ and degree one to the element of $V$. Remark that $(\C^{p|q})_{\bos}=\C^p$. A \emph{complex superdomain} is a superspace which is isomorphic to an open subset of 
$\C^{p|q}$, with the induced structure sheaf.  
Given a homogenous ideal sheaf  $\mathcal{I}$ on a complex superdomain $U \subset \C^{p|q}$, the closed subspace of $U$ defined by $\I$ is the superspace whose underlying topological space is the closed subspace of $U$ defined by the ideal sheaf $(\I+\n_{U})/\n_{U}$ of $\O_U/\n_U$ and whose structure sheaf is the restriction of the sheaf $\O_{U}$. A superspace is called \emph{affine} if it is isomorphic to a closed subspace of $\C^{p|q}$; it is called \emph{quasi-affine} if it is isomorphic to a closed subspace of a complex superdomain.

A \textbf{complex superspace} is a superspace which is locally isomorphic to a quasi-affine superspace. More intrinsically, a superspace $(X,O_X$) is a complex superspace if and only if $(X,\O_{X,0})$ is complex space and $\O_{X,1}$ is a coherent $O_{X,0}$-module (see \cite[Prop. 1.1.3]{Vaintrob}).
Note that bosonic complex superspaces are (ordinary) complex spaces  and the bosonic truncation and quotient of a complex superspace  is a complex space. Moreover, the  infinitesimal neighborhoods of a complex superspace are again complex superspaces and $X^{(n)}=X$ for $n\gg 0$. 

A complex superspace  is smooth, or equivalently it is a \textbf{complex supermanifold},  if
\begin{itemize}
\item $X_{\bos}=(X,\O_X/\n_X)$ is a complex manifold;
\item the sheaf $\n_X/\n_X^2$ is a locally free sheaf of $\O_X/\n_X$-modules;
\item  $\O_X$ is locally  isomorphic to the $\Z/2\Z$-graded exterior algebra $\Lambda^{\bullet}(\n_X/\n_X^2)$ over $\n_X/\n_X^2$.
\end{itemize}
A complex supermanifold is said to be of dimension $p|q$ (for some integers $p,q\geq 0$) if $\dim X_{\bos}=p$ and $\rk(\n_X/\n_X^2)=q$. Note that a connected supermanifold has always  dimension $p|q$, for some (uniquely determined) $p, q\geq 0$, and that a complex supermanifold has dimension $p|q$ if each of its connected components has  dimension $p|q$. Equivalently, a complex superspace is smooth of dimension $p|q$ if and only if it is locally isomorphic to $\C^{p|q}$. 

Observe that bosonic complex supermanifolds (i.e. complex supermanifolds of dimension $p|0$) are (ordinary) complex manifolds, and that the bosonic truncation of a $p|q$ dimensional complex supermanifold is a $p$ dimensional complex manifold. On the other hand, the bosonic quotient $X/\Gamma$ of a complex supermanifold $X$ is a complex space whose reduced structure $(X/\Gamma)_{\rm red}$ is isomorphic to $X_{\bos}$. However, $X/\Gamma$ is rarely a complex manifold (or equivalently is reduced): for example, if $X$ is a complex supermanifold of dimension $p|q$  then $X/\Gamma$ is a complex manifold  if and only if $q\leq 1$.

Given integers $p,q\geq 0$, we have the free sheaf   $\O_X^{p|q}:=\O_X^p|\O_X^q:=\O_X^p\oplus \Pi \O_X^q$ of rank $p|q$, where $\Pi$ is the parity change functor. In other words, $\O_X^{p|q}$ is a free $\O_X$-module having $p$ even generators and $q$ odd generators. A \emph{locally free sheaf} $F$ of rank $p|q$ on a complex superspace is a sheaf of $\O_X$-modules  that is locally isomorphic  to $\O_X^{p|q}$, i.e. for any $x\in X$ there exists an open subset $x\in U\subset X$ such that $F_{|U}\cong \O_U^{p|q}$ (see \cite[Appendix B.3]{Fioresi}). A \emph{coherent sheaf} $F$ on a complex superspace $X$ is  a sheaf of $\O_X$-modules which is locally of finite presentation, i.e. such that for every $x\in X$ there exists an open subset $x\in U\subset X$ and an exact sequence of even homomorphisms $\O_X^{n|m}\to \O_X^{p|q}\to F\to 0$ (see \cite[Sec. 1.3]{Vaintrob}).
We denote by $F_{\bos}$ the pull-back of $F$ to $X_{\bos}$. The sheaf $F_{\bos}$ is a $\mathbb{Z}_2$-graded coherent sheaf of $\O_{X_{\bos}}$-modules, i.e. it admits a splitting 
$F_{\bos}=F_{\bos}^+\oplus F_{\bos}^-$ into an even and an odd coherent subsheaf.   We will sometimes write $F_{\bos}=F_{\bos}^+|F_{\bos}^-$ to denote the splitting of $F_{\rm bos}$ into its even and odd part. 
 If $F$ is locally free of rank $p|q$, then $F_{\bos}=F_{\bos}^+| F_{\bos}^-$, with $F_{\bos}^+$ and $F_{\bos}^-$ locally free of rank, respectively, $p$ and $q$. A line bundle $L$ is a locally free sheaf of rank either $1|0$ or $0|1$.

It is possible to construct a complex superspace  out of a complex space $M$ and a locally free sheaf $E$: the topological space is $M$, and the structure sheaf is the $\Z/2\Z$-graded exterior algebra $\bigwedge^{\bullet} E$. We denote this  complex supermanifold   by $M_E=(M,\bigwedge^{\bullet} E)$. In this case, $(M_E)_{\bos}=M$ and $\n_{M_E}/\n_{M_E}^2=E$. The inclusion $\O_M\hookrightarrow \bigwedge^{\bullet} E$ induces a section of the natural inclusion $M=(M_E)_{\bos} \hookrightarrow M_E$; this section is  called the canonical splitting of $M_E$. 
A complex superspace   $X$ is called \emph{split} \footnote{this is called decomposable in \cite[Chap. 4.4]{Manin1} and it is stronger than being split in the sense of \cite[1.1.1]{Vaintrob}, which  in particular does not imply that $\n_X/\n_X^2$ is locally free.} if it is isomorphic to  $M_E$, for some complex space $M$ and some locally free sheaf $E$ on $M$ (and then necessarily we must have that $M=X_{\bos}$ and $E=\n_X/\n_X^2$).
 Note that a necessary condition for a complex superspace $X$ to be split is  that $\n_X/\n_X^2$ is a locally free sheaf on $X_{\bos}$, which is a rather strong property. By definition, a complex supermanifold is locally split but, in general, not split (see e.g. \cite[Chap. IV.10]{Manin1}). Any complex superspace such that $\n_X/\n_X^2$ is a line bundle on $X_{\bos}$ (for example, any $n|1$ dimensional complex supermanifold)  is split: the proof of  \cite[Chap. IV.8]{Manin1} for $n|1$ supermanifolds works verbatim for the more general case.

The \emph{tangent bundle} $T_X$ of complex superspace  $X$ is the sheaf of derivation of $\O_X$. It is a coherent sheaf and it is locally free if $X$ is smooth. For a split complex supermanifold $M_E$, we have that 
the $(T_{M_E})_{\rm bos}=TM | E^{\vee}$.

Let us now recall the following properties of a morphism $f\colon X\to B$  of complex superspaces. We say that:
\begin{itemize}
\item $f$ is \emph{proper} (resp. \emph{separated}, resp.  \emph{finite}, resp. \emph{surjective}) if the underlying map of topological spaces   is proper (resp. \emph{separated}, resp. \emph{finite}, resp. \emph{surjective}).
\item  $f$ is \emph{\'{e}tale} if, for any $x\in X$, the homomorphism of local superrings $f^{\sharp}_ x:\O_{B,f(x)}\to \O_{X, x}$ is an isomorphism. 
\item $f$ is \emph{smooth} if, for any $x\in X$, there exist a open neighborhood $U$ of $x$ and a open neighborhood $V$ of $f(x)$ such that $f(U)\subseteq V$ and $f_{|U}: U\to V$ is the composition of an \'{e}tale map from $U$ to $V\times \mathbb{C}^{p|q}$, for some $p$ and $q$, and the projection to $V$.
\end{itemize}

An \'{e}tale morphism is smooth: it is enough to take $p=q=0$ in the definition of smooth morphism. Let us stress that if $f$ is a smooth morphism, then for any point $b \in B$ the fibre $f^{-1}(b)$ is smooth. The following Remark is straightforward and it is recorded here for future reference.

\begin{remark}\label{R:incl-et}
If $f:X\to S$ is a morphism of complex spaces which satisfies one of the above properties, then if we interpret $f$ as a morphism of complex superspaces (via the natural inclusion $i$) then it satisfies the corresponding property. 
\end{remark}

The following Proposition is an important direct consequence of the definition of \'etale morphism.

\begin{proposition}\label{P:etale}
If $f:X\to B$ is an \'etale morphism of complex superspaces, then we have that
\noindent
\begin{enumerate}[(i)]
\item \label{P:etale1}  $f_{\bos}:X_{\bos} \to B_{\bos} $ is an \'etale morphism of complex spaces;
\item \label{P:etale2}   $f/\Gamma:X/\Gamma \to B/\Gamma$ is an \'etale morphism of complex spaces.
\end{enumerate}
\end{proposition}
\begin{proof}
By assumption, for any point  $x\in X$, the homomorphism of local superrings $f^{\sharp}_ x:\O_{B,f(x)}\to \O_{X, x}$ is an isomorphism. This implies that 
$f^{\sharp}_ x(\m_{f(x)})=\m_x$  and hence that the induced local homomorphism 
$$(f_{\bos})^{\sharp}_ x:\O_{B_{\bos},f(x)}=\O_{B,f(x)}/\m_{f(x)}\to \O_{X, x}/\m_x=\O_{X_{\bos}, x}$$
is an isomorphism, i.e. that $f_{\bos}$ is \`etale. Moreover, since $f^{\sharp}_ x$ is equivariant with respect to the action of the canonical automorphism $\Gamma$, by taking invariants we get that the induced local   homomorphism 
$$(f/\Gamma)^{\sharp}_ x:\O_{B/\Gamma,f(x)}=\O_{B,f(x)}^{\Gamma}\to \O_{X, x}^{\Gamma}=\O_{X/\Gamma, x}$$
is an isomorphism, i.e. that $f/\Gamma$ is \`etale. 
\end{proof}

We can generalize property \eqref{P:etale1} of the previous Proposition to smooth morphisms. 

\begin{proposition}\label{P:smooth}
If $f:X\to B$ is a smooth morphism of complex superspaces, then  $f_{\bos}:X_{\bos} \to B_{\bos} $ is a smooth morphism of complex spaces.
\end{proposition}
\begin{proof}
By hypothesis, $f$ is locally the composition of an \'etale map $g\colon U\to V\times \mathbb{C}^{p|q}$ and the projection to $V$. Since the bosonic truncation preserves both etalness (by Proposition \ref{P:etale}\eqref{P:etale1}) and fibre products, we deduce that  $f_{\bos}$ is locally the composition of the \'{e}tale map $g_{\bos}\colon U_{\bos}\to (V\times \mathbb{C}^{p|q})_{\bos}=V_{\bos}\times \mathbb{C}^p$ and the projection to $V_{\bos}$; hence $f_{\bos}$ is smooth. 

\end{proof}

The following lemma will be a key ingredient in the study of stacks for the \'etale topology.

\begin{lemma}\label{lem:quoziente}
Let $f\colon X \to S$ and $g\colon Y \to S$ be two morphisms of complex superspaces. There exists a natural morphism $\pi \colon X/\Gamma\times_{S/\Gamma}Y/\Gamma \to (X\times_S Y)/\Gamma$, which is an isomorphism at the level of topological spaces; furthermore, if $f$ is \'etale then $\pi$ is an isomorphism.
\end{lemma}
\begin{proof}
From the quotient maps $X\to X/\Gamma$, $S\to S/\Gamma$ and $Y\to Y/\Gamma$, we get a morphism $\rho\colon X\times_S Y\to X/\Gamma\times_{S/\Gamma}Y/\Gamma$. The codomain is an ordinary complex space, so this morphism has to factor via a morphism $\pi \colon(X\times_S Y)/\Gamma \to  X/\Gamma\times_{S/\Gamma} Y/\Gamma $, whose existence was claimed in the statement.

To show that $\pi$ is an isomorphism at the level of topological spaces, it is enough to prove that the same property for  the morphism of ordinary complex spaces
$$ \rho_{\bos}: X_{\bos}\times_{S_{\bos}} Y_{\bos}=(X\times_S Y)_{\rm bos} \to X\times_S Y\stackrel{\rho}{\longrightarrow} X/\Gamma\times_{S/\Gamma}Y/\Gamma.$$ 

This follows from the fact that the bosonic truncation and the bosonic quotient of a complex superspace have the same underlying topological space, together with the fact that the topological space underlying the fibre product of two complex spaces is the fibre product of their underlying topological spaces (see for example \cite[III.1.6]{GK}).

We now focus on the last part of the lemma. Consider  the induced morphism of structure sheaves 
$$\pi^{\sharp}:\O_{(X\times_S Y)/\Gamma}= (\O_X\otimes_{\O_S}\O_Y)^{\Gamma} \to \O_{X/\Gamma\times_{S/\Gamma}Y/\Gamma}=\O_X^{\Gamma}\otimes_{\O_S^{\Gamma}}\O_Y^{\Gamma},$$
where, with a slight abuse of notation, we identify $\O_X$ with its pull-back to $X\times_S Y$ and similarly for $\O_Y$ and $\O_S$.
For any  point $x\in X$ and $y\in Y$ such that $f(x)=g(y)=z\in S$, consider the induced morphism of local superrings
$$\pi^{\sharp}_{(x,y)}:\O_{(X\times_S Y)/\Gamma,(x,y)}= (\O_{X,x}\otimes_{\O_{S,z}}\O_{Y,y})^{\Gamma} \to \O_{X/\Gamma\times_{S/\Gamma}Y/\Gamma,(x,y)}=\O_{X,x}^{\Gamma}\otimes_{\O_{S,z}^{\Gamma}}\O_{Y,y}^{\Gamma}.$$
If $f$ is \'etale, then $\pi^{\sharp}_{(x,y)}$ is an isomorphism for every $(x,y)\in X\times_S Y$, since $f^{\sharp}_x:\O_{S,z}\to \O_{X,x}$ is an isomorphism. This implies that the morphism $\pi^{\sharp}$ is an isomorphism, and hence that $\pi$ is an isomorphism as required. 

\end{proof}

\begin{remark}\label{R:bad-Gamma}
The bosonic quotient  does not preserve in general neither fibre products nor smoothness: the space $(\mathbb{C}^{0|1}\times \mathbb{C}^{0|1})/\Gamma\cong \C^{0|2}/{\Gamma}=\displaystyle \Spec \C[x]/(x^2)$ is not isomorphic to $(\mathbb{C}^{0|1}/\Gamma)\times (\mathbb{C}^{0|1}/\Gamma)\cong \Spec(\C)$, where the fibre product is taken over $\Spec(\C)$; a projection $\pi$ from $\mathbb{C}^{0|3}$ to $\mathbb{C}^{0|2}$ is smooth, but $\pi/\Gamma$ is not smooth.
\end{remark}

\end{section}

\begin{section}{Superstacks}\label{sec:superstack}

In this section we introduce  complex superstacks  and we discuss their relationship with complex (ordinary) stacks. For the classical theory of (ordinary) stacks, see  \cite{Vis} or \cite[Chap. XII]{ACG} and the references therein.

\subsection{Categories fibered in groupoids}\label{S:CFG}

Let $\S$ be the category of complex superspaces.

\begin{definition}\label{D:CFG}
A \emph{category fibered in groupoids} (or  CFG for short) over $\S$ is a pair $\M=(\cC,p)=(\cC_M,p_M)$ consisting of a category $\cC_M=\cC$ together with a functor (called the fibration functor) 
$$ p_{\M}=p \colon \cC \to \S$$
satisfying the following two axioms:
\begin{enumerate}
\item \label{axiom1} Given a morphism $f\colon T \to S$ in $\S$ and an object $\eta$ in $\cC$ such that $p(\eta)=S$, there exists an object $\xi$ in $\cC$ (called a \emph{pull-back} of $\eta$ along $f$) and a morphism $\phi\colon \xi\to \eta$ in $\cC$ such that $p(\phi)=f$. 
\item \label{axiom2} All morphisms $\phi\colon \xi\to \eta$ in $\cC$ are cartesian in the following sense. Given a morphism $\phi'\colon \xi'\to \eta$ in $\cC$ and a morphism $h\colon p(\xi')\to p(\xi)$ in $\S$ such that $p(\phi)\circ h =p(\phi')$, there exists a unique morphism $\psi\colon \xi'\to \xi$ in $\cC$ such that $h=p(\psi)$ and $\phi\circ \psi=\phi'$.
\end{enumerate}
\end{definition}
CFG's over $\S$ form a $2$-category, denoted by $\FIB_{\S}$,  in which $1$-morphisms are functors commuting with the fibration and $2$-morphisms between morphisms $F,G:\M\to \M'$ are natural transformation of functors $t:F\to G$ such that $p_{\M'}(t_{\eta})=\id_{p_{\M}(\eta)}$ for any $\eta\in \cC_{\M}$ (see  \cite[\href{http://stacks.math.columbia.edu/tag/02XS}{Tag02XS}]{Stacks}). 

The $2$-category $\FIB_{\S}$ admits \emph{fibre products}  (see \cite[\href{http://stacks.math.columbia.edu/tag/0041}{Tag 0041}]{Stacks}): given two morphisms $f_i:\M_i\to \N$ of CFG's over $\S$,  the fibre product of $\M_1\times_{\N} \M_2 $ is the CFG over $\S$ which is defined as it follows. The objects of the underlying category are quadruples $(S,\eta_1, \eta_2, \phi)$ with $S\in \S$,  $\eta_i\in \cC_{\M_i}(S)$ and $\phi: f_1(\eta_1) \stackrel{\cong}{\longrightarrow} f_2(\eta_2)$ is an isomorphism in $\N(U)$. A morphism between $(S,\eta_1, \eta_2, \phi)$ and $(S',\eta_1', \eta_2', \phi')$ is a pair $(\mu_1,\mu_2)$ where $\mu_i: \eta_i\to \eta_i'$ is a morphism in $\M_i$ such that $p_{\M_1}(\mu_1)=p_{\M_2}(\mu_2)$ and $f_2(\mu_2)\circ \phi= \phi'\circ f_1(\mu_1) $. The fibration functor  $p_{\M_1\times_{\N} \M_2}: \cC_{\M_1\times_{\N} \M_2} \to \S$ sends an object $(S,\eta_1, \eta_2, \phi)$ onto $S$. 
The projections $p_i: \M_1\times_{\N} \M_2 \to \M_i$ are the obvious forgetful morphisms and there is a natural transformation of functors $\Psi:f_1\circ p_1\to f_2\circ p_2$ which is is given on an object $\eta=(S, \eta_1, \eta_2,\phi)$ by $\Psi_{\eta}:=\phi:(f_1\circ p_1)(\eta)=f_1(\eta_1)\to f_2(\eta_2)=(f_2\circ p_2)(\eta)$. 
 The $2$-category $\FIB_{\S}$ admits a final object, which is given by $\S=(\S,\id_{\S})$. The fibre product of two CFG's $\M_1$ and $\M_2$, denoted by $\M_1\times \M_2$, is the fibre product of the two morphisms $\M_1\to \S$ and $\M_2\to \S$.

\vspace{0.1cm}

Given a CFG $\M$ over $\S$ and an object $S\in \S$, the \emph{fiber} of $\M$ over $S$ is the subcategory  $\M(S)$ of $\cC_{\M}$ whose objects are the objects $\eta$ of $\cC_{\M}$ such that $p(\eta)=S$, and the morphisms are the morphisms $\phi$ in $\cC$ such that $p_{\M}(\phi)=\id_S$. It follows from  axiom \eqref{axiom2} that $\M(S)$ is a groupoid, i.e. a category in which every morphism is invertible. 
Given a morphism $f:T\to S$ in $\S$ and an object $\eta\in \M(S)$, we choose a pull-back of $\eta$ along $f$ as in  \eqref{axiom1} which we denote by $f^*(\eta)$;
the collection of all such choices is called a cleavage of $\M$ and we implicitly assume that all our CFG are endowed with a cleavage.
Note that pull-backs are unique up to a unique isomorphisms by axiom \eqref{axiom2} so that most of the constructions that we are going to perform turns out to be independent of the choice of a cleavage. 
After the choice of a cleavage has been made, axiom \eqref{axiom2} guarantees that the association 
$\eta\in \M(S)\mapsto f^*(\eta)\in \M(T)$ is a functor, called a  \emph{pull-back functor} (it depends on the choice of a cleavage). Note however, that given two morphisms $V\stackrel{g}{\longrightarrow} T \stackrel{f}{\longrightarrow} S$ in $\S$, we will in general have that $(f\circ g)^*\neq g^*\circ f^*$ (for any choice of cleavage).
 In other words, the association $S\mapsto \M(S)$ is in general not a functor. However the association $S\mapsto \M(S)$   is  a pseudofunctor from the category $\S$ to the  $2$-category of groupoids, and this pseudofunctor recovers the original CFG together with the chosen cleavage  (see \cite[Sec. 3.1]{Vis} for a discussion of this equivalence).

\vspace{0.1cm}

We can associate to any complex superspace $X$ a CFG $\un X=(\cC_X,p_X)$ over $\S$ as follows. The objects of $\cC_X$ are morphisms $(f:S\to X)$ in $\S$ and the morphisms between $(f:S\to X)$ and $(f':S'\to X)$ are the morphisms $g:S\to S'$ in $\S$ such that $f=f'\circ g$. The fibration functor $p_X$ sends $(f:S\to X)$ onto $S$. Note that the fiber of $\un X$ over $S$ is the set of morphisms $\Hom(S,X)$ (regarded as  a trivial groupoid in which the only arrows are the identities) and a pull-back functor along $f:S\to T$ is the map $f^*=-\circ f :\Hom(T,X)\to \Hom(S,X)$ obtained by precomposing with $f$ (note that this defines a canonical  cleavage on $\un X$). 
Given any $X\in \S$ and any $\M\in \FIB_{\S}$, there is an equivalence of categories (called the \emph{$2$-Yoneda lemma}, see \cite[Sec. 3.6]{Vis}) 
\begin{equation}\label{E:Yoneda}
\begin{aligned}
\Hom_{\FIB_{\S}}(\un X, \M) & \stackrel{\cong}{\longrightarrow} \M(X), \\
\mu & \mapsto \mu(\id_X), \\
\end{aligned}
\end{equation}
whose inverse is obtained by sending $\eta\in \M(X)$ into the morphism $\mu_{\eta}$ which sends an object $(f:S\to X)$ of $\un X$ onto $f^*(\eta)\in \M(Y).$ In particular, when $\M=\un Y$ we get an equivalence of categories
\begin{equation}\label{E:fullembed}
\Hom_{\FIB_{\S}}(\un X, \un Y)  \stackrel{\cong}{\longrightarrow} \Hom_{\S}(X,Y).
\end{equation}
In other words, the association $X\mapsto \un X$ defines a fully faithful embedding of $\S$ into $\FIB_{\S}$.  

\vspace{0.1cm}

In a similar way, we can define the $2$-category $\FIB_{\S_{\rm ev}}$ of categories fibered in groupoids over the category $\S_{\rm ev}$ of even complex superspaces (i.e ordinary complex spaces) and the fully faithful embedding of $\S_{\rm ev}$ into $\FIB_{\S_{\rm ev}}$. 

\vspace{0.1cm}

The $2$-categories $\FIB_{\S}$ and $\FIB_{\S_{\rm ev}}$ are related by a pair of adjoint $2$-functors that  are constructed starting from the 
bosonic  quotient $/\Gamma: \S\to \S_{\ev}$ and the natural inclusion $i:\S_{\rm ev}\to \S$ via  fibre product of categories (see \cite[Tag 0040]{Stacks}). More precisely, the first $2$-functor (called the \textbf{natural inclusion}) is given by 
\begin{equation}\label{E:func-i}
\begin{aligned}
i: \FIB_{\S_{\rm ev}} & \longrightarrow \FIB_{\S}, \\
\N & \mapsto i(N):=(\cC_{\N}\times_{\S_{\rm ev}} \S, p_2) 
\end{aligned}
\end{equation} 
where the tensor product of categories is done with respect to the fibration functor $p_{\N}:\cC_{\N}\to \S_{\rm ev}$ and the bosonic quotient $/\Gamma:\S\to \S_{\rm ev}$, and $p_2$ is the projection onto the second factor. Explicitly, the objects of $i(\N)$ are triples $(S,\eta,f)$ where $S\in \S$, $\eta\in \cC_{\N}$ and $f$ is an isomorphism in $\S_{\rm ev}$ between $S/\Gamma$ and $p_{\N}(\eta)$; the morphisms are the natural ones.  In particular, the fiber of $i(\N)$ over $S\in \S$ is equal to $i(\N)(S)=\N(S/\Gamma)$. From the fact that the bosonic quotient $/\Gamma:\S\to \S_{\rm ev}$ is left adjoint to the natural inclusion $i:\S_{\rm ev}\to \S$, it follows easily that the morphism $i$ of  \eqref{E:func-i} extends the natural inclusion  $i:\S_{\rm ev}\to \S$, or in other words that $i(\un Y)=\un{i(Y)}$ for any complex space $Y\in \S_{\rm ev}$.

The second $2$-functor (called the \textbf{bosonic truncation}) is given by 
\begin{equation}\label{E:funct-bos}
\begin{aligned}
(-)_{\bos}: \FIB_{\S} & \longrightarrow \FIB_{\S_{\rm ev}}, \\ 
\M  & \mapsto \M_{\bos}:=(\cC_{M}\times_{\S} \S_{\rm ev}, p_2),
\end{aligned}
\end{equation}
 where the tensor product of categories is done with respect to the fibration functor $p_{\M}:\cC_{\M}\to \S$ and the natural inclusion 
 $i:\S_{\rm ev}\to \S$, and $p_2$ is the projection onto the second factor. 
 Explicitly, the objects of $\M_{\bos}$ are triples $(S,\eta,f)$ where $S\in \S_{\rm ev}$, $\eta\in \cC_{\M}$ and $f$ is an isomorphism in $\S$ between $i(S) $ and $p_{\M}(\eta)$; the morphisms are the natural ones.  In particular, the fiber of $\M_{\bos}$ over $S\in \S_{\rm ev}$ is equal to $\M_{\bos}(S)=\M(i(S))$. From the fact that the natural inclusion  $i:\S_{\rm ev}\to \S$ is left adjoint of the bosonic truncation 
 $(-)_{\bos}: \S\to \S_{\rm ev}$, it follows easily that the morphism $(-)_{\bos}$ of  \eqref{E:funct-bos} extends the bosonic truncation  $(-)_{\bos}:\S\to \S_{\rm ev}$, or in other words that $(\un X)_{\bos}=\un{X_{\bos}}$ for any complex superspace $X\in \S$. 

The $2$-functors $i$ and $(-)_{\bos}$ satisfy the following properties: 

\begin{enumerate}[(i)]
\item The natural inclusion $i$ exhibits $\FIB_{\S_{\rm ev}}$ as the full sub-$2$-category of $\FIB_{\S}$ consisting of all the elements $\M\in \FIB_{\S}$ such that $\M(S)=\M(i(S/\Gamma))$ for any $S\in \S$.
\item Given $\M\in \FIB_{\S}$, the bosonic truncation $\M_{\bos}$ can be identified with the full subcategory of $\M$ consisting of all objects $\eta\in \cC_{\M}$ such that $p_{\M}(\eta)\in \S_{\rm ev}$, endowed with the restriction of the fibration functor $p_{\M}$. 
\item The natural inclusion $i$ is left adjoint of the bosonic truncation $(-)_{\bos}$. In particular, for any $\M\in \FIB_{\S}$ there is a morphism $i(\M_{\bos})\to \M$ of CFG over $\S$ which is universal with respect to morphisms from CFG over $\S_{\rm ev}$ (regarded as CFG over $\S$ via the natural inclusion) into $\M$. 
\end{enumerate}

\begin{remark}\label{R:NObosquot}
We do not know how to extend the bosonic quotient functor $/ \Gamma:\S \to \S_{\rm ev}$ to CFG's. The reason is that $/\Gamma$ does not preserve fibre products (see Remark \ref{R:bad-Gamma}), so it does not admit a left adjoint functor. However, in Theorem \ref{T:DMtrunc} we will  extend the bosonic quotient functor from DM complex superstacks to DM complex stacks, using atlases.
\end{remark}

\subsection{Superstacks}\label{S:stacks}

In order to define superstacks (which is the aim of this subsection), we will need to recall the descent theory for CFG's over $\S$ (or over $\S_{\rm ev}$) with respect to the \emph{\'etale topology}, i.e. the Grothendieck topology on $\S$ whose coverings are \'etale surjective morphisms of complex superspaces.

Let $\M$ be a CFG over $\S$ and let $\xi\in \M(T)$.  Let $p\colon T\to S$ be an \'{e}tale morphism of complex superspaces, and denote by $p_i$ (with $i=1,2$) the projections from $T\times_S T$ onto the $i$-th factor, by $p_{ij}$ (with $1\leq i<j \leq 3$) the projections from $T\times_S T\times_S T$ onto the $(i,j)$-factor, and by $q_i$ (with $i=1,2,3$) the projections 
from $T\times_S T\times_S T$ onto the $i$-th factor.
 A \emph{descent datum} for $\xi$ along the \'etale morphism $p:T\to S$ is an isomorphism
$$\phi \colon p_1^*\xi \to p_2^*\xi$$
satisfying the cocycle condition $p_{23}^*\xi\circ p_{12}^*\xi=p_{13}^*\xi:q_1^*\eta\to q_3^*\xi$. 
A descent datum $\xi$ along $p:T\to S$ is called \emph{effective} if there exists an object $\eta\in \M(S)$ and an isomorphism $\psi \colon p^*\eta\to \xi$ such that $\phi=(p_2^*\psi)\circ(p_1^*\psi)^{-1}$.

\begin{definition}[Superstack] \label{D:superstack}
A stack in groupoids over $\S$ for the \'etale topology (or simply a \emph{superstack}) is a category fibered in groupoids  $\M$ over $\S$ that satisfies the following two conditions 
\begin{enumerate}
\item \label{stack1} Every descent datum (for the \'etale topology) is effective.
\item \label{stack2} For any $S \in \S$, and any pair of objects $\eta$ and $\eta'$ in $\M(S)$, the contravariant functor (called the \emph{isomorphism functor} between $\eta$ and $\eta'$)
$$\Isom_S(\eta,\eta'): \Hom(-,S) \to {\rm Sets}$$
that associate to every  $f:T\to S$ the set of isomorphisms in $\M(T)$ between $f^*(\eta)$ and $f^*(\eta')$  is a sheaf in the \'{e}tale topology, i.e. for every surjective \'etale morphism $\pi:X\to Y$ of complex superspaces over $S$  the diagram
$$
\xymatrix{
\Isom_S(\eta,\eta')(Y) \ar[r]^{m} & \Isom_S(\eta,\eta')(X) \ar@<-.5ex>[r]_(0.4){m_2} \ar@<.5ex>[r]^(0.4){m_1} & \Isom_S(\eta,\eta')(X\times_Y X)
}
$$
is exact, where $m:=\Isom_S(\eta,\eta')(\pi)$ and $m_i:=\Isom_S(\eta,\eta')(\pi_i)$ with $\pi_i:X\times_Y X\to X$ is the projection onto the $i$-th factor.
\end{enumerate}
\end{definition}

The $2$-category of superstacks is the full sub $2$-category of $\FIB_{\S}$  whose objects are superstacks  (while $1$-morphisms and $2$-morphisms are the same as the ones of $\FIB_{\S}$). The $2$-category of superstacks  admits fibre products  since it is easily checked that the fibre product of two morphisms of superstacks  (seen as CFG's over $\S$)  is a superstack.


In a similar way, one can define stacks in groupoids over $\S_{\rm ev}$ with respect to the \'etale topology, which we call  (even or bosonic) \emph{stacks}. They form a full sub-$2$-category of $\FIB_{\S_{\rm ev}}$ which is closed under taking fibre products. 

\begin{lemma}\label{L:i-bos-stacks}
\noindent 
\begin{enumerate}[(i)]
\item \label{L:i-bos-stacks1}The natural inclusion $i: \FIB_{\S_{\rm ev}} \to \FIB_{\S}$ sends stacks in superstacks. 
\item \label{L:i-bos-stacks2} The bosonic truncation $(-)_{\bos}: \FIB_{\S} \to \FIB_{\S_{\rm ev}}$ sends superstacks into stacks.
\end{enumerate}
\end{lemma} 
\begin{proof}
Part \eqref{L:i-bos-stacks1}: let $\N$ be a stack and let us prove that $i(\N)\in \FIB_{\S}$ is a superstack by checking that conditions \eqref{stack1} and \eqref{stack2} of Definition  \ref{D:superstack} are satisfied. 

Consider first a descent datum for $\xi\in i(\N)(T)$  with respect to the \'etale morphism $p:T\to S$ of complex superspaces. 
By applying Lemma \ref{lem:quoziente}, this induces a descent datum for $\xi\in \N(T/\Gamma)=i(\N)(T)$ with respect to the  morphism $p/\Gamma:T/\Gamma\to S/\Gamma$ of complex spaces, which is \'etale by Proposition \ref{P:etale}\eqref{P:etale2}. Since $\N$ is a stack, this descent datum is effective, i.e. there exists $\eta\in \N(S/\Gamma)$ and an isomorphism $\psi:  (p/\Gamma)^*\eta\to \xi$ in $\N(T/\Gamma)$ inducing the descent datum for $\xi$ with respect to $p/\Gamma$. 
But then, if we interpret $\eta$ as an element of $i(\N)(S)=\N(S/\Gamma)$ and $\psi$ as an isomorphism between $p^*\eta$ and $\eta$ in $i(\N)(T)$, $\eta$ and $\psi$ induce the descent datum for $\eta$ with respect to $p:T\to S$, and thus condition \eqref{stack1} is satisfied.  
 
Consider next the isomorphism functor $\Isom_S(\eta, \eta')$ between two objects $\eta$ and $\eta'$ in $i(\N)(S)$. By the definition of $i(\N)$, the functor $\Isom_S(\eta, \eta')$  is given by the composition of  $-/\Gamma:\Hom_{\S}(-,S)\to \Hom_{\S_{\ev}}(-/\Gamma,S/\Gamma)$ with the isomorphism functor $\Isom_{S/\Gamma}(\eta,\eta')$ between the same objects $\eta$ and $\eta'$ considered as objects of $\N(S/\Gamma)=i(\N)(S)$. Using that $-/\Gamma$ preserves \'etale morphisms by Proposition \ref{P:etale}\eqref{P:etale2} and fibre products by Lemma \ref{lem:quoziente}, and that $\Isom_{S/\Gamma}(\eta,\eta')$ is a sheaf for the 
\'etale topology on $\S_{\ev}$ since $\N$ is a stack, we deduce that $\Isom_S(\eta, \eta')$ is a sheaf for the \'etale topology on $\S$, and thus condition \eqref{stack2} is satisfied.

The proof of part \eqref{L:i-bos-stacks2} is similar using that for any $\M\in \FIB_{\S}$ the fibers of its bosonic truncation $(\M)_{\bos}$ are equal to $(\M)_{\bos}(S)=\M(i(S))$ and that the natural inclusion functor $i:\S_{\rm ev}\to \S$ preserves fiber products and it sends \'etale morphism of complex spaces into \'etale morphisms of complex superspaces (see Remark \ref{R:incl-et}).

\end{proof}

\vspace{0.1cm}

An important and commonly used criterion to check that condition \eqref{stack2} is satisfied for a CFG over $\S$,  is to show that $\Isom(\eta, \eta')$ is representable, i.e. it is the \emph{functor of points} $\Hom(-, X)$ of a complex superspace $X$ over $S$, and then using the following:

\begin{lemma}\label{L:repres}
Let $X$ be a complex superspace. 
\begin{enumerate}[(i)]
\item \label{L:repres1} The functor of points of $X$
$$\begin{aligned}
h_X:= \Hom(-,X): \S & \longrightarrow {\rm Sets} \\
T & \mapsto \Hom(T,X)
\end{aligned}
$$
is a sheaf for the \'etale topology.
\item \label{L:repres2} The CFG $\un X$ associated to $X$ is a superstack.
\end{enumerate}
\end{lemma}
\begin{proof}
The proof of part \eqref{L:repres1} proceeds along the same lines of the proof for ordinary complex spaces (see e.g. \cite[Chap. XII, Thm. 7.4]{ACG} or \cite[Thm. 2.55]{Vis}) and it boils down to the following result of supercommutative algebra: if $f: S\to T$ is a surjective flat (e.g. \'etale) morphism of affine complex superspaces
then the sequence of supercommutative $\C$-superalgebras 
$$
\xymatrix{
0\to \Gamma(T, \O_T) \ar[r]^{f^{\sharp}} &  \Gamma(S, \O_S) \ar@<-.5ex>[r]_(0.3){e_2} \ar@<.5ex>[r]^(0.3){e_1} &  \Gamma(S, \O_S)\times_{\Gamma(T, \O_T)} \Gamma(S, \O_S)
}
$$
is exact, where $e_1(a)=a\otimes 1$ and $e_2(a)=1\otimes a$. 

Part \eqref{L:repres2} follows from the fact that $\un X$ is the CFG associated to the functor of points $h_X$  as in \cite[Prop. 3.26]{Vis} and the fact that $h_X$ is a sheaf for the \'etale topology (see \cite[Prop. 4.9]{Vis}). 
\end{proof}

In particular, the above Lemma together with the $2$-Yoneda lemma says that the functor $X\mapsto \un X$ is a fully faithful embedding of $\S$ into the $2$-category of superstacks. In what follows, we will say that a superstack is a complex superspace if it lies on the essential image of this embedding, i.e. if it is isomorphic to $\un X$ for some complex superspace. The above Lemma holds true also for the category $\S_{\rm ev}$, so that we get an embedding of the category of complex spaces into the $2$-category of stacks.

\subsection{Complex superstacks}\label{S:algstacks}

In order to introduce complex superstacks (which is the aim of this subsection), we need to be able to extend the notions of separated/proper/smooth/\'etale from morphisms between complex superspace to certain morphism of superstacks defined as it follows. 

A morphism $f:\M\to \N$ between superstacks is called \emph{representable} if for every complex superspace $S$ and every morphism $\un S\to \N$, the fibre product $\M\times_{\N} \un S$ is a complex superspace. Given any property $\textbf{P}$ of morphisms of complex superspace which is stable under base change (like separated, proper, smooth, \'etale, surjective), 
we say that a representable morphism $f:\M\to \N$ of superstacks satisfies property $\textbf{P}$ if for any morphism $\un S\to \N$ from a complex superspace, the morphism between complex superspaces $\M\times_{\N} \un S \to \un S$ satisfies property $\textbf{P}$.

We will be interested in superstacks $\M$ for which every morphism  from a complex superspace onto $\M$ is representable. This property can be  characterized as follows.   

\begin{lemma}\label{L:repr-diag}
Let $\M$ be a superstack. Then the following properties are equivalent:
\begin{enumerate}[(i)]
\item \label{L:repr-diag1} every morphism $\un S\to \M$ from a complex superspace is representable;
\item \label{L:repr-diag2} the diagonal $\Delta:\M\to \M\times \M$ is  representable; 
\item \label{L:repr-diag3} for any complex superspace $S$  and  any pair of objects $\eta$ and $\eta'$ in $\M(S)$, the isomorphism functor $\Isom_S(\eta,\eta')$ is the $S$-relative functor of points of a complex superspace over $S$, i.e. there exists a complex superspace $I(\eta, \eta')$ together with a map $\phi:I(\eta, \eta')\to S$ such that  $\Isom_S(\eta,\eta')(T\stackrel{f}{\longrightarrow}S)$ is the set of all morphisms $\wt f:T\to I(\eta,\eta')$ such that $f=\phi\circ \wt f$. 
\end{enumerate}
\end{lemma}
\begin{proof}
The proof of the equivalence of \eqref{L:repr-diag1} and \eqref{L:repr-diag2} is the same as the one of \cite[Sec. XII, Lemma 8.1]{ACG} using that for every morphisms 
$f:S\to \M$ and $g:T\to \M$ we have that 
$$S\times_{\M} T\cong \M\times_{\M\times \M} (S\times T).$$ 

The equivalence of \eqref{L:repr-diag2} and \eqref{L:repr-diag3} follows from the fact that if we have a morphism $S\to \M\times\M$ whose two projections are given by the elements  $\xi\in \M(S)$ and $\eta\in \M(S)$, then for every $T\to S$ we have (as in \cite[Chap. XII, Equ. (8.4)]{ACG}):
$$ (\M \times_{\M\times \M} S)(T)=\Isom_S(\xi, \eta)(T).$$
\end{proof}

Given any property $\textbf{P}$ of morphisms of complex superspace which is stable under base change (like separated, proper, smooth, \'etale, surjective), 
we say that a representable morphism $f:\M\to \N$ of superstacks satisfies property $\textbf{P}$ if for any morphism $\un S\to \N$ from a complex superspace, the morphism between complex superspaces $\M\times_{\N} \un S \to \un S$ satisfies property $\textbf{P}$.

We are now ready to define (DM) complex   superstacks.

\begin{definition}[(DM) complex superstack]\label{def:DMstack} 
A superstack $\M$ is called a \emph{complex  superstack} (resp. a \emph{DM(=Deligne-Mumford) complex superstack})  if:
\begin{enumerate}
\item \label{def:DMstack1} the diagonal $\Delta\colon \M \to \M\times \M$ is representable  and separated;
\item \label{def:DMstack2} there exists a complex superspace  $X$, and a smooth (resp. \'etale) surjective morphism $\un X\to \M$, which is called an (resp. \'etale) \emph{atlas}.
\end{enumerate}
Given a complex superstack $\M$, we say that:
\begin{enumerate}[(i)]
\item \label{prop1} $\M$ is \emph{smooth} if there exists a smooth atlas, or equivalently if any atlas is smooth. 
\item \label{prop2} $\M$ is \emph{separated} if the diagonal $\Delta:\M\to \M\times \M$ is proper. 
\end{enumerate}

Given a smooth DM complex superstack $\M$, we say that $\M$ has \emph{dimension} $p|q$ if there exists an  \'etale atlas of $\M$ which is smooth of dimension $p|q$, or equivalently if any \'etale atlas $\M$  is smooth of dimension $p|q$.
\end{definition}

Note that the definition of dimension for a smooth DM complex superstacks is well posed since if $f:X\to Y$ is an \'etale morphism of complex supermanifolds 
then the dimension of $X$ is equal to the dimension of $Y$ (this follows directly from the definition of \'etale morphisms).

Combining Lemma \ref{L:repres} and Lemma \ref{L:repr-diag}, we get the following remark which will be useful in what follows. 

\begin{remark}\label{scorciatoia}
Let $\M$ be a CFG over $\S$ and suppose  that for any $S\in \S$ and any two objects $\eta, \eta'\in \M(S)$, the isomorphism functor $\Isom_S(\eta,\eta')$ is the $S$-relative functor of points of a complex superspace over $S$ (resp. an $S$-separated complex superspace). Then condition \eqref{stack2} of Definition \ref{D:superstack} (resp. condition \eqref{def:DMstack1} of Definition \ref{def:DMstack})  is satisfied. 
\end{remark}

The $2$-category of complex superstacks, which we will denote by $\sSt$, is the full sub $2$-category of the $2$-category of superstacks (hence also of the $2$-category $\FIB_{\S}$)  whose objects are complex superstacks  (while $1$-morphisms and $2$-morphisms are the same as the ones of $\FIB_{\S}$). The $2$-category of complex superstacks  admits fibre products  since it is easily checked that the fibre product of two morphisms of complex superstacks  (seen as CFG's over $\S$)  is a complex superstack.
The $2$-category of DM complex   superstacks, which we will denote by $\esSt$, is the full sub $2$-category of $\sSt$  whose objects are DM complex   superstacks. Similarly, one defines the $2$-category $\St$ of complex stacks, and its full sub $2$-category $\eSt$ of DM complex   stacks. 

The superstack $\un X$ associated to a complex superspace is DM complex   superstack since the diagonal $\Delta:\un X\to \un X\times \un X=\un{X\times X}$ is representable (being a morphism of complex superspaces) and separated (because this is true for complex spaces) and an atlas for $\un X$ is just the identity map $\un X\to \un X$. Therefore, using the $2$-Yoneda lemma, the association $X\mapsto \un X$ is a fully faithful embedding of the category $\S$ of complex superspaces into the $2$-category $\esSt$ of DM complex   superstacks. And similarly, there is a fully faithful embedding of the category $\S_{\ev}$ of complex spaces into the $2$-category $\eSt$ of DM complex stacks. \emph{From now on, we will identify a complex (super)space $X$ with its associated DM complex   (super)stack $\un X$.}

The natural inclusion and the the bosonic truncation extend easily from complex (super)spaces to complex (super)stacks.

\begin{proposition}\label{P:i-bos-DM}
\noindent 
\begin{enumerate}[(i)]
\item \label{P:i-bos-DM1} The natural inclusion $i: \FIB_{\S_{\rm ev}} \to \FIB_{\S}$ sends complex (resp. DM, resp. separated, resp. smooth) stacks into complex  (resp. DM, resp. separated, resp. smooth) superstacks. 

\item \label{P:i-bos-DM2} The bosonic truncation $(-)_{\bos}: \FIB_{\S} \to \FIB_{\S_{\rm ev}}$ sends complex  (resp. DM, resp.  separated, resp.  smooth) superstacks into complex  (resp. DM, resp. separated, resp. smooth)  stacks.

\end{enumerate}
\end{proposition} 

\begin{proof}
Part \eqref{P:i-bos-DM1}: let $\N$ be a complex stack and consider $i(\N)\in \FIB_{\S}$. Note that $i(\N)$ is a superstack by Lemma \ref{L:i-bos-stacks}\eqref{L:i-bos-stacks1}. Let us show that $i(\N)$ is a complex superstack by checking that conditions \eqref{def:DMstack1} and \eqref{def:DMstack2} of Definition \ref{def:DMstack} are satisfied.  In doing this, we will also prove that $i(\N)$ inherits the same properties (e.g. DM, smoothness, separatedness) of $\N$.

In order to prove that condition \eqref{def:DMstack1} is satisfied, it is enough to check, according to Remark \ref{scorciatoia}, that for any   $\eta, \eta'\in i(\N)(S)$, the 
isomorphism functor $\Isom_S(\eta, \eta')$ is the $S$-relative functor of an $S$-separated complex superspace $I(\eta, \eta')$. 
By the definition of $i(\N)$, the functor $\Isom_S(\eta, \eta')$  is given by the composition of  $-/\Gamma:\Hom_{\S}(-,S)\to \Hom_{\S_{\ev}}(-/\Gamma,S/\Gamma)$ with the isomorphism functor $\Isom_{S/\Gamma}(\eta,\eta')$ between the same objects $\eta$ and $\eta'$ considered as objects of $\N(S/\Gamma)=i(\N)(S)$.
Since $\N$ is a complex stack, we have that $\Isom_{S/\Gamma}(\eta,\eta')$ is the $S/\Gamma$-relative functor of points of an $S/\Gamma$-separated complex space $\ov I(\eta, \eta')$. Then  the fact that $\Isom_S(\eta, \eta')=\Isom_{S/\Gamma}(\eta,\eta')\circ -/\Gamma$ implies that  $\Isom_S(\eta, \eta')$ is the $S$-relative functor of points of the $S$-separated complex space $I(\eta, \eta'):=S\times_{i(S/\Gamma)} i(\ov I(\eta, \eta'))$, as required. 
Moreover, notice that if $\N$ is separated then $\ov I(\eta, \eta')\to S/\Gamma$ is proper, which implies that  $I(\eta, \eta')\to S$ is proper (because they have the same underlying   topological spaces) and hence that $i(\N)$ is separated. 

In order to check condition \eqref{def:DMstack2}, let $Y_0\to \N$ be a smooth (resp. \'etale) atlas of $\N$. Since $i$ preserve smooth and \'etale maps (see Remark \ref{R:incl-et}), the map $i(Y_0)\to i(\N)$ is a smooth (resp. \'etale) atlas, as required. Observe also that if $\N$ is smooth, then any such atlas $Y_0$ is smooth, which implies that $i(Y_0)$ is smooth, and hence that $i(\N)$ is smooth as well.

The proof of part \eqref{P:i-bos-DM2} is similar to the proof of part  \eqref{P:i-bos-DM1} using that the bosonic truncation preserves fiber products, \'etaleness of morphisms  (see Proposition \ref{P:etale}\eqref{P:etale1}), smoothness of morphism (see Proposition \ref{P:smooth}), and it does not change the underlying topological spaces. We will only prove that if $\M$ is a DM complex  superstack then $\M_{\bos}$ satisfies condition \eqref{def:DMstack1} of Definition \ref{def:DMstack}, and we will leave the remaining verifications to the reader. 

In order to prove that condition \eqref{def:DMstack1} is satisfied for $\M_{\bos}$, it is enough to check, according to Remark \ref{scorciatoia}, that for any   $\eta, \eta'\in \M_{\bos}(T)$, the isomorphism functor $\Isom_T(\eta, \eta')$ is the $T$-relative functor of a $T$-separated complex superspace $I(\eta, \eta')$. 
By the definition of $\M_{\bos}$, the functor $\Isom_T(\eta, \eta')$  is given by the composition of  $i:\Hom_{\S_{\ev}}(-,T)\to \Hom_{\S}(i(-),i(T))$ with the isomorphism functor $\Isom_{i(T)}(\eta,\eta')$ between the same objects $\eta$ and $\eta'$ considered as objects of $\M(i(T))=\M_{\bos}(T)$.
Since $\M$ is a complex superstack, we have that $\Isom_{i(T)}(\eta,\eta')$ is the $i(T)$-relative functor of points of a $i(T)$-separated complex superspace $\ov I(\eta, \eta')$. Then  the fact that $\Isom_T(\eta, \eta')=\Isom_{i(T)}(\eta,\eta')\circ i(-)$ implies that  $\Isom_T(\eta, \eta')$ is the $T$-relative functor of points of the $T$-separated complex space $I(\eta, \eta'):= i(\ov I(\eta, \eta'))_{\bos}$, as required. 

\end{proof}

Let us finish this section with the definition of coarse complex superspace of a complex superstack.

\begin{definition}[Coarse superspace]\label{D:coarse}
Let $\M$ be a complex superstack. A complex superspace  $M$ together with a map $p\colon \M \to M$ is a \emph{coarse superspace} for $\M$ if
\begin{enumerate}
\item \label{D:coarse1} the map $p$ is bijective on $\C$-valued points;
\item \label{D:coarse2} given any other complex superspace  $Z$ and a map $f\colon\M\to Z$, the map $f$ factors trough $p$.
\end{enumerate}
When it exists, we will denote the coarse superspace for $\M$ (which is unique by property \eqref{D:coarse2}) by $|\M|$.
\end{definition}

If $\M$ is equal to a complex superspace  $X$, then clearly $X$ itself is its own coarse superspace. 
In a similar way, we can define the coarse space of a complex stack. A classical result of Keel and Mori \cite{KM} asserts that any separated DM complex stack admits a coarse space. It would be interesting to extend this result to separated DM complex  superstacks.

\subsection{Complex Supergroupoids}\label{SS:supergrp}

There is a close relation between complex superstacks and complex supergroupoids, which we are going to explain in this subsection. 
See also \cite{lerman} and \cite{lie} for some nice surveys on this relation.

\begin{definition}[(\'Etale) complex supergroupoids]\label{D:supergrp}
A separated smooth (resp. \'etale) complex supergroupoid, or simply a \emph{complex supergroupoid} (resp. an  \emph{\'etale complex supergroupoid}), is the datum of two complex superspaces $X_0$ and $X_1$ and two smooth (resp. \'etale) morphisms $s,t:X_1\to X_0$ (called source and target) such that $(s,t):X_1\to X_0\times X_0$ is separated, together with a
unit morphism  $u:X_0\to X_1$, an inverse morphism $\iota:X_1\to X_1$ and a composition morphism $m:X_1 {}_{s}{\times}_{t}X_1 \to X_1$ satisfying the following relations
\begin{equation}\label{E:ax-grpd}
\begin{aligned}
& (s,t)\circ u=\Delta_{X_0}, \hspace{0.5cm} (s,t)\circ \iota= \eta_{X_0}\circ (s,t), \hspace{0.5cm} (s,t)\circ m=(s\times t), \\
& m\circ (u\circ t, \id_{X_1}) =\id_{X_1}, \hspace{0.5cm}m\circ (\id_{X_1}, u\circ s)=\id_{X_1}, \\
& m\circ (\id_{X_1}, m)=m\circ (m, \id_{X_1}),
\end{aligned}
\end{equation}
where $\Delta_{X_0}:X_0\to X_0\times X_0$ is the diagonal of $X_0$ and $\eta_{X_0}:X_0\times X_0\to X_0\times X_0$ interchanges the two factors.
A complex supergroupoid will be denoted by 
$$\left(\xymatrix{
X_1\ar@<-.5ex>[r]_{t}\ar@<.5ex>[r]^{s} & X_0\\
}, u, \iota, m\right)
\text{ or  by } 
\xymatrix{
X_1\ar@<-.5ex>[r]_{t}\ar@<.5ex>[r]^{s} & X_0\\
} \text{ or by }X_1\rightrightarrows X_0 \textrm{  or simply by  } X_{\bullet} \,.$$
\end{definition}

Complex supergroupoids  form a $2$-category, which we will denote by $\sGr$:  
\begin{itemize}
\item a $1$-morphism (or simply a morphism) 
$$f=(f_1,f_0)\colon \left(\xymatrix{
X_1\ar@<-.5ex>[r]_{t}\ar@<.5ex>[r]^{s} & X_0\\
}, u, \iota, m\right) \to \left(\xymatrix{
X_1'\ar@<-.5ex>[r]_{t'}\ar@<.5ex>[r]^{s'} & X_0'\\
}, u', \iota', m'\right) $$ 
is a pair of maps $f_i\colon X_i\to X_i'$, with $i=0,1$, such that the natural compatibility conditions hold:
\begin{equation}\label{E:comp-1mor}
f_0\circ s=s'\circ f_1, \hspace{0.2cm} f_0\circ t=t'\circ f_1, \hspace{0.2cm} f_1\circ u=u'\circ f_0, \hspace{0.2cm} f_1\circ \iota= \iota'\circ f_1, \hspace{0.2cm} f_1\circ m=m'\circ (f_1\times f_1);
\end{equation}
\item a $2$-morphism $\alpha$ between two morphisms $f=(f_1,f_0),g=(g_1,g_0)\colon X_{\bullet}\to X'_{\bullet}$ as above is a map $\alpha\colon X_0\to X_1'$ such that 
\begin{equation}\label{E:comp-2mor}
\begin{sis} 
& s(\alpha(x))=f_0(x) \text{ and } t(\alpha(x))=g_0(x) \text{ for every } x\in X_0,\\
&  m(g_1(a),\alpha(s(a)))=m(\alpha(t(a)),f_1(a)) \text{ for every } a\in X_1.
\end{sis}
\end{equation}
\end{itemize}

The $2$-category of \'etale complex   supergroupoids, which we will denote by $\esGr$, is the full sub $2$-category of $\sGr$ whose objects are \'etale complex   supergroupoids. 
Similarly, one defines the $2$-category $\Gr$ of complex groupoids and its full sub $2$-category $\eGr$ of \'etale complex   groupoids.

\vspace{0.1cm}

To any complex (rep. \'etale) supergroupoid $X_1\rightrightarrows X_0$, we can associate (as in \cite[p. 303-304]{ACG}) 
a complex (resp. DM) superstack, denoted by $[X_1\rightrightarrows X_0]$, whose fiber over a complex superspace $T$ is given by 
\begin{equation}\label{E:as-superstack}
[X_1\rightrightarrows X_0](T)=\lim_{T'\to T}\Hom(T'\times_T T' \rightrightarrows T',X_1\rightrightarrows X_0 )\,.
\end{equation}
where the (inverse) limit runs over all the surjective \'etale morphisms $T'\to T$ and $\Hom(T'\times_T T' \rightrightarrows T,X_1\rightrightarrows X_0 )$ is the category whose objects consists of pairs $(\Phi_0,\Phi_1)$ of morphisms $\Phi_0:T\to X_0$ and $\Phi_1:T'\times_T T'\to X_1$ satisfying the obvious compatibility relations. 
It is easily checked that $1$ and $2$-morphisms of complex supergroupoids induce $1$ and $2$-morphisms of the associated complex superstacks; hence we get a $2$-functor 
\begin{equation}\label{E:funcF}
\begin{aligned}
\F: \sGr &\longrightarrow \sSt,\\
X_1\rightrightarrows X_0 & \mapsto [X_1\rightrightarrows X_0],
\end{aligned}
\end{equation}
which sends $\esGr$ into $\esSt$. Similarly, one defines a $2$-functor 
$$\F_{\rm ev}: \Gr\to \St$$
which sends $\eGr$ into $\eSt$.  

\vspace{0.1cm}

Notice that the complex superstacks that are in the image of $\F$ comes with a canonical morphism  $X_0=[\xymatrix{X_0\ar@<-.5ex>[r]_{\id}\ar@<.5ex>[r]^{\id} & X_0}]\to [X_1\rightrightarrows X_0]$, which is surjective and smooth (resp.   \'etale if and only if the complex supergroupoid $X_1\rightrightarrows X_0$ is \'etale), i.e. it is an atlas of 
$[X_1\rightrightarrows X_0]$.
Conversely, given an atlas  $X_0\to \M$ of a complex superstack $\M$, the fiber product $X_1:=X_0\times_{\M} X_0$ is a complex superspace, the two projections (called source and target) $s,t:X_1\to X_0$ are smooth (and \'etale if and only if $\M$ is DM) and the composite map $(s,t):X_1\to X_0\times X_0$ is separated. Moreover there are the following morphisms: the unit morphism  $u:X_0\to X_1$ (the diagonal), the inverse $\iota:X_1\to X_1$ (which interchanges the two factors) and the composition morphism $m:X_1 {}_{s}{\times}_{t}X_1=X_0\times_{\M} X_0\times_{\M} X_0 \to X_1=X_0\times_{\M} X_0$ (which is the projection onto the first and third factor), which satisfy the relations \eqref{E:ax-grpd}. In order words, starting with a complex (resp. DM) superstack and an atlas $X_0\to \M$, we have constructed a complex (resp. \'etale) supergroupoid $X_1:=X_0\times_{\M}X_0\rightrightarrows X_0$, and  we have an isomorphism 
\begin{equation}\label{E:grp-pres}
\M=[X_1\rightrightarrows X_0],
\end{equation} 
which is called a \textbf{supergroupoid presentation} of $\M$ (clearly it depends on the chosen atlas $X_0\to \M$).

Observe that, given a supergroupoid presentation of $\M$ as in \eqref{E:grp-pres},  we have that $\M$ is smooth if and only if $X_0$ is smooth, and $\M$ is separated if and only if 
$(s,t):X_1\to X_0\times X_0$ is proper.

\vspace{0.1cm}

The natural inclusion and the bosonic truncation can be extended from complex (super)stacks to complex (super)groupoids.

\begin{proposition}\label{P:i-bosDM}
\noindent 
\begin{enumerate}[(i)]
\item \label{P:i-bosDM1} The association 
$$\begin{aligned}
i: \Gr & \longrightarrow \sGr, \\
\left(\xymatrix{
Y_1\ar@<-.5ex>[r]_{t}\ar@<.5ex>[r]^{s} & Y_0\\
}, u, \iota, m\right)  & \mapsto \left(\xymatrix{
i(Y_1)\ar@<-.5ex>[r]_{i(t)}\ar@<.5ex>[r]^{i(s)} & i(Y_0)\\
}, i(u), i(\iota), i(m)\right)
\end{aligned}
$$ 
defines a $2$-functor which sends $\eGr$ into $\esGr$ and such that $i([Y_1\rightrightarrows Y_0])=[i(Y_1)\rightrightarrows i(Y_0)]$.



\item \label{P:i-bosDM2} The association 
$$\begin{aligned}
(-)_{\bos}: \sGr & \longrightarrow \Gr, \\
\left(\xymatrix{
X_1\ar@<-.5ex>[r]_{t}\ar@<.5ex>[r]^{s} & X_0\\
}, u, \iota, m\right)
& \mapsto \left(\xymatrix{
(X_1)_{\bos}\ar@<-.5ex>[r]_{t_{\bos}}\ar@<.5ex>[r]^{s_{\bos}} & (X_0)_{\bos}\\
}, u_{\bos}, \iota_{\bos}, m_{\bos} \right)
\end{aligned}
$$ 
defines a $2$-functor which sends $\esGr$ into $\eGr$ and such that $[X_1\rightrightarrows X_0]_{\bos}=[(X_1)_{\bos}\rightrightarrows (X_0)_{\bos}]$.

\item \label{P:i-bosDM3} The $2$-functor $i$ is $2$-fully faithful and it is left adjoint of the $2$-functor $(-)_{\bos}$.

\end{enumerate}
\end{proposition} 

\begin{proof}
\eqref{P:i-bosDM1}: consider a separated smooth (resp. \'etale) complex groupoid $Y_1\rightrightarrows Y_0$ with maps $(s,t,u, \iota, m)$ satisfying \eqref{E:ax-grpd}. Since $i$ preserves fibre products, \'etaleness or smoothness  and it does not change the underlying topological spaces, the pair of complex superspaces   $i(Y_1), i(Y_0)$ together with maps $(i(s),i(t),i(u), i(\iota), i(m))$ satisfies  \eqref{E:ax-grpd}, the morphism $(i(s),i(t)):i(Y_1)\to i(Y_0)\times i(Y_0)$ is separated and the maps $i(s)$ and $i(t)$ are smooth (resp. \'etale); hence it defines a separated smooth (resp. \'etale) complex supergroupoid $i(Y_1)\rightrightarrows i(Y_0)$. The fact that $i$ is a $2$-functor follows easily from the definition of $1$-morphisms and $2$-morphisms together with the above mentioned properties of $i$. 

Let $\N:=[Y_1\rightrightarrows Y_0]$ be the complex stack associated to  $Y_1\rightrightarrows Y_0$. Since $i$ preserves smooth maps and fiber products, the map $i(Y_0)\to i(\N)$ is smooth and it holds that $i(Y_0)\times_{i(\N)}i(Y_0)=i(Y_0\times_{\N} Y_0)=i(Y_1)$; this implies that $i(\N)=[i(Y_1)\rightrightarrows i(Y_0)]$, as required.

The proof of part \eqref{P:i-bosDM2} is similar to the proof of part  \eqref{P:i-bos-DM1} using that the bosonic truncation preserves fiber products, \'etaleness of morphisms  (see Proposition \ref{P:etale}\eqref{P:etale1}), smoothness of morphism (see Proposition \ref{P:smooth}), and it does not change the underlying topological spaces. 

Part \eqref{P:i-bosDM3}: the fact that $i$ is $2$-fully faithful, i.e. that for any two objects $Y_{\bullet}, Y'_{\bullet}\in \Gr$ we have an equivalence of categories 
$i(-):\Hom(Y_{\bullet}, Y'_{\bullet})\to \Hom(i(Y_{\bullet}), i(Y'_{\bullet}))$, follows from the fact that $i$ defines a fully faithful embedding of $\S_{\ev}$ into $\S$ together with the fact that $i$
preserves fibre products.  The fact that $i$ is right adjoint of the functor $(-)_{\bos}$, i.e. that for any $X_{\bullet}\in \sGr$ and any $Y_{\bullet}\in \Gr$ there is an equivalence of categories
$\Hom_{\sGr}(i(Y_\bullet), X_\bullet,)\cong \Hom_{\Gr}(Y_\bullet, (X_\bullet)_{\bos})$, follows  from the fact that $i: \S_{\ev}\to \S$ is left adjoint of $(-)_{\bos}:\S\to \S_{\ev}$ 
 together with the fact that $i$ and $(-)_{\bos}$ both preserve fibre products. 

\end{proof}

The aim of the remaining part of this section  is to define the bosonic quotient. We will define it for \'etale complex supergroupoids and then descent it to DM complex superstacks.
In order to do this, we have to study more careful the properties of the $2$-functor $\F: \esGr\to \esSt$ of \eqref{E:funcF}  (and of the analogous $2$-functor $\F_{\rm ev}:\eGr\to \eSt$). 

First of all, since every DM complex  superstack admits a supergroupoids  presentation as in \eqref{E:grp-pres}, we deduce that the $2$-functor $\F$ is essentially surjective. 
However, the $2$-functor $\F$ is not an equivalence of $2$-categories since there are some non-invertible morphisms of \'etale complex supergroupoids which become invertible, i.e. equivalences, at the level of DM complex superstacks. Indeed, it turns out that  the $2$-functor $\F$ realizes $\esSt$ as a suitable localization of $\esGr$, as we are now going to recall.

\vspace{0.1cm}

In \cite{Pro} (generalizing the construction of \cite{GZ} for $1$ categories), it is explained how to define the {\bf localization} $\mathcal{C}[W^{-1}]$ of a 2-category $\mathcal{C}$ at a convenient collection of 1-morphisms $W$, namely those that admit
a right calculus of fractions (see \cite[Sec. 2.1]{Pro}). Let us recall that:
\begin{itemize}
\item  The objects of $\mathcal{C}[W^{-1}]$ are the objects of $\mathcal{C}$.
\item A $1$-morphism of $\cC[W^{-1}]$ between two objects  $A$ and $B$ is a diagram of $1$-morphisms of $\cC$ (sometimes called a \emph{roof})
$$
A \xleftarrow{w} C\xrightarrow{f} B
$$
such that  $w\in W$. For the composition of $1$-morphisms, see \cite[Sec. 2.2]{Pro}.
\item A $2$-morphism of $\cC[W^{-1}]$ between two $1$-morphisms $A \xleftarrow{w_1} C_1\xrightarrow{f_1} B$ and $A \xleftarrow{w_2} C_2\xrightarrow{f_2} B$
is a quadruple $(u_1,u_2,\alpha,\beta)$, where  $u_1:E\to C_1$ and $u_2:E\to C_2$ are $1$-morphisms of $\cC$ such that $w_1\circ u_1, w_2\circ u_2\in W$
while $\alpha: w_1\circ u_1\Rightarrow w_2\circ u_2$ and $\beta:f_1\circ u_1\Rightarrow f_2\circ u_2$ are $2$-morphisms of $\cC$.
Such a quadruples are considered up to a natural equivalence relation, see \cite[Sec. 2.3]{Pro}. The horizontal and vertical compositions of $2$-morphisms are defined in \cite[Sec. 2.3]{Pro}.
\end{itemize}
The localization $\cC[W^{-1}]$ is not in general a 2-category, but only a bicategory: this means that the  associativity and unity laws of horizontal composition hold only up to natural isomorphisms (and indeed $\cC[W^{-1}]$ can be constructed even starting with a bicategory $\cC$).

There is a pseudofunctor (i.e. a functor which preserves composition and identities only up to natural isomorphisms) of bicategories 
$U:\cC\to \cC[W^{-1}]$ which is defined by:
\begin{equation*}
\begin{sis}
& U(A)=A  \text{ for any object } A \text{ of } \cC,\\
& U(f)=(A \xleftarrow{\id_A} A\xrightarrow{f} B)   \text{ for any $1$-morphism } f:A\to B \text{ of } \cC,\\
& U(\alpha)=(\id_A, \id_A, \id_{\id_A}, \alpha)  \text{ for any $2$-morphism } \alpha:f\Rightarrow g \text{ of } \cC.
\end{sis}
\end{equation*}
The homomorphism $U$ sends $1$-morphisms in $W$ into equivalences of $\cC[W^{-1}]$ and it is  universal  among the pseudofunctors of bicategories $\cC\to \cD$ that send $1$-morphisms in $W$ into equivalences of $\cD$ (see \cite[Sec. 3]{Pro}).

\vspace{0.1cm}

Let us go back to \'etale complex (super)groupoids. 
A morphism $f=(f_1,f_0)\colon X_{\bullet}\to Y_{\bullet}$ between two \'etale complex (super)groupoids is called a \emph{weak equivalence} if $f_0\colon X_0 \to Y_0$ and $f_1:X_1\to Y_1$ are \'etale and surjective,
and 
$$s\times f_1 \times t  \colon X_1 \to  X_0\, {}_{f_0}\times_{s}\, Y_1\,  {}_{t}   \times_{f_0} \,X_0$$
is an isomorphism of complex (super)spaces. We denote by $W$ (resp. $W_{\ev}$) the collection of all weak equivalences among \'etale complex supergroupoids (resp. groupoids). 

Mimic the proof of \cite[Sec. 4.1]{Pro} for \'etale topological groupoids, it can be shown that $W$ and $W_{\ev}$ admits a right calculus of fraction. Therefore, we can form the localizations 
$\esGr[W^{-1}]$ and $\eGr[W_{\ev}^{-1}]$.  

Arguing as in \cite[Thm. 3]{Moe} (where the author deals with \'etale topological groupoids), one can show that if $f\in W$ (resp. $f\in W_{\ev}$) then $\F(f)$ (resp. $\F_{\ev}(f)$) is an equivalence of complex supergroupoids (resp. groupoids). 
Therefore, from the universal property of the localization recalled above, the strict $2$-functors $\F$ and $\F_{\ev}$ factors through pseudofunctors of bicategories
\begin{equation}\label{E:loc-mor}
\F_{\loc} \colon \esGr[W^{-1}] \to \DMs \hspace{0.3cm} \text{ and } \hspace{0.3cm} (\F_{\ev})_{\loc}: \sGr[W_{\ev}^{-1}]\to \DM.
\end{equation}

Adapting the proof of \cite[Sec. 7]{Pro} for  DM algebraic stacks (which is similar to the proof for topological stacks in Sec. 5 of loc. cit. and differentiable stacks in Sec. 6 of loc. cit.), it can be shown that $\F_{\loc}$ and $(\F_{\ev})_{\loc}$ are isomorphism of bicategories.

\vspace{0.1cm}

We have now all the ingredients we need to introduce the bosonic quotient for DM complex superstacks by descending it from the bosonic quotient at the level of \'etale complex supergroupoids. In what follows, etalness will play a key role via Lemma \ref{lem:quoziente}.

\begin{theorem}\label{T:DMtrunc}
\noindent 
\begin{enumerate}[(i)]
\item \label{T:DMtrunc1} 
The association 
$$\begin{aligned}
-/\Gamma: \esGr & \longrightarrow \eGr, \\
\left(\xymatrix{
X_1\ar@<-.5ex>[r]_{t}\ar@<.5ex>[r]^{s} & X_0\\
}, u, \iota, m\right)
& \mapsto \left(\xymatrix{
X_1/\Gamma\ar@<-.5ex>[r]_{t/\Gamma}\ar@<.5ex>[r]^{s/\Gamma} & X_0/\Gamma\\
}, u/\Gamma, \iota/\Gamma, m/\Gamma\right)
\end{aligned}
$$ 
defines a $2$-functor which is left adjoint to the natural inclusion $i:\eGr\to \esGr$.


\item \label{T:DMtrunc2} The bosonic truncation defined above preserves weak equivalences, hence it gives a pseudo-functor
$$
\begin{aligned}
/\Gamma \colon \DMs & \longrightarrow \DM \\
\M=[X_1\rightrightarrows X_0] & \mapsto \M/\Gamma:=[X_1/\Gamma\rightrightarrows X_0/\Gamma]
\end{aligned}
$$
which is left adjoint to the natural inclusion $i:\DM \to \DMs$.  Moreover, a DM complex  superstack $\M$ is separated if and only if $\M/\Gamma$ is separated. 
\end{enumerate}
\end{theorem} 
Note that the bosonic quotient functor does not preserve smoothness (see Remark \ref{R:bad-Gamma}). Let us also stress once more that the bosonic truncation introduced in \eqref{T:DMtrunc1} preserves the composition of 1-morphisms and the identities; whereas the bosonic truncation introduced in  \eqref{T:DMtrunc2} preserves the composition of $1$-morphisms and identities only up to natural isomorphisms, because we need to pass trough the localization at the weak equivalences.

\begin{proof}
Part \eqref{T:DMtrunc1}.
Consider an \'etale separated complex supergroupoid $\left(\xymatrix{
X_1\ar@<-.5ex>[r]_{t}\ar@<.5ex>[r]^{s} & X_0\\
}, u, \iota, m\right)$ and let us show that $\left(\xymatrix{
X_1/\Gamma\ar@<-.5ex>[r]_{t/\Gamma}\ar@<.5ex>[r]^{s/\Gamma} & X_0/\Gamma\\
}, u/\Gamma, \iota/\Gamma, m/\Gamma\right)$ is an \'etale separated complex groupoid. 
The map $(s/\Gamma,t/\Gamma):X_1/\Gamma\to X_0/\Gamma\times X_0/\Gamma$ is separated since the bosonic 
quotient does not change the underlying topological spaces. The maps $s/\Gamma$ and $t/\Gamma$ are \'etale since the bosonic  quotient preserves \'etale maps by Proposition \ref{P:etale}\eqref{P:etale2}.  Therefore, we are left with checking the axioms \eqref{E:ax-grpd}. The last three axioms follow from the fact that the bosonic quotient preserves fiber products as soon as one of the maps is \'etale, see Lemma \ref{lem:quoziente}. In order to prove the first three axioms for 
$X_1/\Gamma\rightrightarrows X_0/\Gamma$, we first apply the bosonic quotient functor to the corresponding three axioms for $X_1\rightrightarrows X_0$ and we get 
\begin{equation}\label{E:equa1}
(s,t)/\Gamma\circ u/\Gamma=\Delta_{X_0}/\Gamma, \hspace{0.5cm} (s,t)/\Gamma\circ \iota/\Gamma= \eta_{X_0}/\Gamma\circ (s,t)/\Gamma, \hspace{0.5cm} (s,t)/\Gamma\circ m/\Gamma=(s\times t)/\Gamma.
\end{equation}
Consider the natural map $\pi\colon (X_0\times X_0)/\Gamma \to X_0/\Gamma\times X_0/\Gamma$ introduced in Lemma \ref{lem:quoziente}. By functoriality, we clearly have that 
\begin{equation}\label{E:equa2}
\pi\circ \Delta_{X_0}/\Gamma=\Delta_{X_0/\Gamma}, \hspace{0.5cm} \pi\circ  (s,t)/\Gamma=(s/\Gamma,t/\Gamma), \hspace{0.5cm} \pi\circ \eta_{X_0}=\eta_{X_0/\Gamma}\circ \pi, \hspace{0.5cm} \pi\circ (s\times t)/\Gamma=(s/\Gamma\times t/\Gamma). 
\end{equation}
Post composing with $\pi$ the relations in \eqref{E:equa1} and using the identities in \eqref{E:equa2}, we get the desired first three axioms for $X_1/\Gamma\rightrightarrows X_0/\Gamma$.

Given a $1$-morphism $f=(f_1,f_0):X_{\bullet}\to X_{\bullet}'$,  the morphisms $f_i/\Gamma:X_i/\Gamma\to X_i'/\Gamma$ for $i=0,1$ satisfy the relations \eqref{E:comp-1mor}: the fist four identities are obvious and the last one follows from the fact that  $X_1 {}_{s}{\times}_{t}X_1/\Gamma=X_1/\Gamma {}_{s/\Gamma}{\times}_{t/\Gamma}X_1/\Gamma$, and similarly for $X_1'$, by Lemma \ref{lem:quoziente} which can be applied because $s$ (and also $t$) are \'etale. Hence we get a  $1$-morphism $f/\Gamma:=(f_1/\Gamma,f_0/\Gamma):X_{\bullet}/\Gamma\to X'_{\bullet}/\Gamma$. 

Given a $2$-morphism $\alpha\colon X_0\to X_1'$ between two $1$-morphisms $f=(f_1,f_0),g=(g_1,g_0)\colon X_{\bullet}\to X'_{\bullet}$, the morphism $\alpha/\Gamma\colon X_0/\Gamma\to X_1'/\Gamma$ satisfies the relations \eqref{E:comp-2mor}: the first two identities are obvious and the last one  again from  follows from the fact that  $X_1 {}_{s}{\times}_{t}X_1/\Gamma=X_1/\Gamma {}_{s/\Gamma}{\times}_{t/\Gamma}X_1/\Gamma$ and the similar fact for $X_1'$. Hence we get a $2$-morphism $\alpha/\Gamma:X_0/\Gamma\to X_1'/\Gamma$ 
between the two $1$-morphisms $f/\Gamma,g/\Gamma\colon X_{\bullet}/\Gamma\to X'_{\bullet}/\Gamma$. 

Finally, the fact that $-/\Gamma$ is left adjoint of  $i$, i.e. that for any $X_{\bullet}\in \esGr$ and any $Y_{\bullet}\in \eGr$ there is an equivalence of categories
$\Hom_{\eGr}(X_\bullet/\Gamma,Y_\bullet)\cong \Hom_{\esGr}(X_\bullet, i(Y_\bullet))$, follows  from the fact that $-/\Gamma: \S\to \S_{\ev}$ is left adjoint of $i:\S_{\rm ev}\to \S$ 
 together with the fact that $i$  preserve fibre products and $-/\Gamma$ preserves fiber products in which at least one of the maps is \'etale by Lemma \ref{lem:quoziente}.

Part \eqref{T:DMtrunc2}: let us show that the bosonic quotient preserves weak equivalences. Let $f=(f_1,f_0):X_\bullet\to Y_\bullet$ a weak equivalence of \'etale complex supergroupoids. 
Since $f_0$ and $f_1$ are \'etale, we get that $f_0/\Gamma$ and $f_1/\Gamma$ are \'etale by Proposition \ref{P:etale}\eqref{P:etale2}.
By Lemma  \ref{lem:quoziente}, we have a  canonical isomorphism 
$$
( X_0 {}_t\times_{f_1} Y_1 {}_{f_1}   \times_s X_0 )/\Gamma \cong    X_0/\Gamma {}_t\times_{f_1} Y_1/\Gamma {}_{f_1}   \times_s X_0/\Gamma.
$$
This implies that $s/\Gamma\times f_1/\Gamma\times t/\Gamma$ is equal to $(s\times f_1\times t)/\Gamma$, and hence that is an isomorphism.  We conclude that $f/\Gamma:X_{\bullet}/\Gamma\to Y_{\bullet}/\Gamma$ is a weak equivalence of \'etale complex groupoids.  

Using that the bosonic quotient preserves weak equivalences and the universal property of the localization recalled above, we infer that the bosonic quotient descends to a pseudofunctor 
$$-/\Gamma: \esGr[W^{-1}] \longrightarrow \eGr[W_{\ev}^{-1}].$$
Using that the functors $\F_{\loc}$ and $(\F_{\ev})_{\loc}$ of \eqref{E:loc-mor} are equivalences of bicategories, we deduce the existence of the required bosonic truncation pseudofunctor 
$$/\Gamma \colon \DMs  \longrightarrow \DM. $$
The fact that this bosonic quotient pseudofunctor $-/\Gamma$ is left adjoint of the natural inclusion $i:\DM\to \DMs$ follows from the analogous statement at the level of \'etale complex (super)groupoids, which was proved in \eqref{T:DMtrunc1}.

Finally, $\M=[X_1\rightrightarrows X_0]$ is separated if and only if $(s,t):X_1\to X_0\times X_0$ is proper, similarly for  $\M/\Gamma:=[X_1/\Gamma\rightrightarrows X_0/\Gamma]$. Again because of Lemma \ref{lem:quoziente}, since $ \pi\circ  (s,t)/\Gamma=(s/\Gamma,t/\Gamma)$ and both the bosonic truncation and the morphism $\pi\colon (X_0\times X_0)/\Gamma \to X_0/\Gamma\times X_0/\Gamma$ do not change the underlying topological spaces, we have that $(s,t):X_1\to X_0\times X_0$ is proper if and only if $(s/\Gamma,t/\Gamma):X_1/\Gamma\to X_0/\Gamma\times X_0/\Gamma$ is proper.

\end{proof}


\begin{remark} It is possible to define a  canonical automorphism $\Gamma_{\c}$  on a category fibered in groupoids $(\c,\F)$. Given an object $\eta$, we define $\Gamma(\eta)$ as the pull-back of $\eta$ via the automorphism $\Gamma$ of the complex superspace $\F(\eta)$. When the category fibered in groupoids comes from a groupoid, then the action of $\Gamma$ is the natural action on the presentation.

At least in the case of Deligne-Mumford superstacks, one could use a standard argument to define the quotient of $\c$ by $\Gamma_{\c}$, see for instance \cite{Rom}. However, given an ordinary complex space $X$, the quotient of the complex superstack $\underline{i(X)}$ by $\Gamma$ is not isomorphic isomorphic to $\underline{i(X)}$: for instance, if $X$ is a point, the quotient is $B\mathbb{Z}_2$. This means that this sort of bosonic quotient is not the left adjoint of the natural inclusion $i$ from DM complex stacks to DM complex superstacks; in particular, it is different from the definition that we gave above, and we will not discuss it further in this paper.

\end{remark}

\end{section}

\section{Susy curves}\label{sec:susy}

\subsection{The definition}

 Let us start with the definition of $1|1$ curves and susy\footnote{susy stands for supersymmetric; sometimes, they are called SUSY$_1$ (as in \cite{Manin2}) or super Riemann surfaces (as in \cite{DW1}).} curves over a fixed complex superspace $S$. 
  
 \begin{definition}\label{def-susy}
 \noindent 
 \begin{enumerate} 
 \item  A \textbf{$1|1$ curve}  over $S$ is a complex superspace $\c$ endowed with a proper smooth morphism
$
\pi \colon \c \to S
$
of relative dimension $1|1$. In particular, the fibers of $\pi$ (over $\C$-points of $S$) are proper complex supermanifolds whose bosonic truncations are compact Riemann surfaces of a certain genus $g$, which is called  the genus of $\c$.
\item A \textbf{susy curve} $(\c,\D)$ over $S$ is a   $1|1$ curve $\c$ over $S$ endowed with a rank $0|1$ maximally non-integrable distribution $\D\subset T_{\c/S}$; maximally non-integrable means that the map  
\begin{equation}\label{bracket-iso}
\begin{aligned}
[-,-]: \D^2 & \longrightarrow T_{\c/S}/\D, \\
t_1\otimes t_2 & \mapsto [t_1,t_2] \mod \D
\end{aligned}
\end{equation}
is an isomorphism. In particular, we get an exact sequence of locally free sheaves 
\begin{equation}\label{E:exactD}
0\to \D \to T_{\c/S}\to \D^2\to 0.
\end{equation}
\end{enumerate}
\end{definition}

We should think to $1|1$ (resp. susy) curves over $S$ as a family of $1|1$ (resp. susy) curves parametrized by $S$. For this reason, 
coordinates on $S$ are sometimes called the moduli of the family; in particular, odd coordinates on $S$ are called the odd moduli of $\c$.
Sometimes we denote a susy curve $(\c,\D)$ just by $\c$.

Given two susy curves $(\c,\D)$ and  $(\c',\D')$ over, respectively, $S$ and $S'$, a morphism $F: (\c,\D)\to (\c',\D')$ is a commutative diagram of morphisms of complex superspaces
\begin{displaymath}
\begin{array}{ccc}
\c  &\xrightarrow{F} &\c' \\
\downarrow&&\downarrow \\
S&\xrightarrow{f}  &S'
\end{array}
\end{displaymath}
such that $dF(\D)=\D'$. In this way, we get a category of susy curves over $S$, whose objects are susy curves over $S$ and whose morphisms are morphisms of susy curves.  

In the special case where $(\c,\D)$ and  $(\c',\D')$ are two susy curves over the same base $S$, we introduce a slightly stronger notion: a morphism $F: (\c,\D)\to (\c',\D')$ \emph{over $S$} is a commutative diagram as above in which $f=\id_S$.

We should often be using the fact (see \cite[Lemma 3.1]{DW1}) that locally at every point of a susy curve $(\c, \D)$ over $S$ we can choose coordinates $z$ and $\theta$ (called  \emph{superconformal coordinates}) such that  $z$ is a local relative coordinate for the underlying families of Riemann surfaces
and the distribution $\D$ is locally generated  by the odd vector field
$$
\frac{\partial}{\partial\theta}+\theta \frac{\partial}{\partial z}.
$$
Note that, in such coordinates, $[\D,\D]$ is generated by the even vector field $\displaystyle \frac{\partial}{\partial z}$.

Given a $1|1$ curve $\c$ over $S$, its bosonic truncation $C:=\c_{\bos}$ is a curve over $S_{\bos}$ (i.e. a smooth family of compact Riemann surfaces). On the other other hand, the 
pullback $\c_{\spli}:=\c\times_S S_{\bos}$  is smooth of relative dimension $1|1$ over a bosonic base; hence $L:=\n_{\c_{\spli}}/\n_{\c_{\spli}}^2$  is a line bundle over $C$ and we have that $\c_{\spli}=C_L$.
Following \cite{Giulio}, we say that $\c_{\spli}$ is the split model of $\c$ (some authors say that $\c_{\spli}$ is obtained by ``turning off the odd moduli"). More generally, for any $k\geq 1$, we denote by $\c_k$ the pull-back of $\c$ to $S^{(k)}$: this is  a susy curve over the $k$-th infinitesimal thickening of $S_{\bos}$ in $S$. Note that $\c_1=\c_{\spli}$.

\subsection{Relation with spin curves}\label{S:susy-spin}
Recall that a spin curve $(C,L, \phi)$ over a complex (bosonic) space  $S$ is a curve $C$ over  $S$, together with a line bundle $L$ on $C$ and an isomorphism
$
\phi\colon K_{C/S} \stackrel{\cong}{\longrightarrow} L\otimes L\,.
$
On a spin curve we always have an action of $\mu_2$, which acts as the identity on $C$ and $\phi$ and as $-1$ on the line bundle $L$. The following is a reformulation of standard results; here we would like to stress the role of the canonical automorphism $\Gamma$.

\begin{proposition}\label{susy_spin}
Let $S$ be an ordinary complex space.  There exists an equivalence of categories between spin curves over $S$ and susy curves over $S$. Under this equivalence, the action of $\mu_2$ is mapped to the canonical automorphism $\Gamma$.
\end{proposition}

\begin{proof}
First, we associate a susy curve to a spin curve. Given a spin curve $(C,L,\phi)$, we define $\c$ to be the split manifold $C_L$. To define locally $\D$, pick a local trivialization $\theta$ of $L$ and a local co-ordinate $z$ such that $\phi(dz)=\theta\otimes \theta$. Now we let $\D$ to be the span of $\frac{\partial}{\partial\theta}+\theta \frac{\partial}{\partial z}$. It is standard to check that this definition makes sense globally and satisfies all the requested properties.

We now do the converse, associating a spin curve to a susy curve. Let $\c$ be a susy curve over an even base $S$. The total space $\c$ has odd dimension equal to one, so it is split, and we can write $\c=C_L$; we thus have a curve $C$ and a line bundle $L$, and to give a spin curve we just need to construct the isomorphism $\phi$. Pick conformal co-ordinates and let  $v=\frac{\partial}{\partial\theta}+\theta \frac{\partial}{\partial z}$ be a generator of $\D$. With this local description, we can see that $\D_{\bos}=L^{-1}$, and $[v,v]=\frac{\partial}{\partial z}$. The bracket of vector fields thus gives the requested isomorphism
$$
\phi^{\vee} \colon L^{-1}\otimes L^{-1} =\D_{\bos}\otimes\D_{\bos}\to T_{\pi}
$$ 

The statement about the automorphism follows because the non trivial element of $\mu_2$ acts as $-1$ on $\D$. This constructions also identifies morphisms of susy curves with morphisms of spin curves.
\end{proof}

In the following, we will denote by $C_L$ the susy curve associated to the spin curve $(C,L,\phi)$ over a complex space $S$, according to Proposition \ref{susy_spin}.

\subsection{The isomorphism functor}\label{S:automorphism}

Given two susy curves $\cC$ and $\cC'$ over a complex superspace $S$, consider the controvariant functor (called the \emph{isomorphism functor} between $\c$ and $\c'$) 
$$\un{\Isom}_S(\cC,\cC'):\Hom(-,S)\to {\rm Sets}$$
that associates to any morphism $T\to S$ the set   $\un \Isom_S(\c, \c')(T)$ consisting of all $T$-isomorphisms from $\c_T:=\c\times_S T$ to $\c'_T:=\c'\times_S T$. 


\begin{proposition}\label{prop:auto}
Given two susy curves $\cC$ and $\cC'$ of genus $g\geq 2$ over a complex superspace $S$, the isomorphism functor $\un{\Isom}_S(\cC,\cC')$ is represented by an  $S$-complex superspace $\phi: \Isom_S(\cC,\cC')\to S$ having the property that $\phi$ is finite and with injective differential at any point.   
\end{proposition}
\begin{proof}

First, we show that $\un{\Isom}_S(\cC,\cC')$ is represented by a  complex superspace $\phi: \Isom_S(\cC,\cC')\to S$. 
We argue similarly to the classical case. Any isomorphism of susy curves preserves the Berenzinian line bundle. By the main result of \cite{Giulio}, the fifth (or higher) power of  this line bundle is very ample, and the dimension of the space of sections depends only on the genus of the susy curves. We conclude that $\un{\Isom}_S(\cC,\cC')$ is represented by a 
 closed subspace of $\Isom_S(E,E')$, where $E$ (resp. $E'$) is the locally free sheaf  on $S$ defined as the push-forward of the fifth power of the Berenzinian line bundle of $\c$ (resp. $\c'$). 

Let us now show that $\phi$ is finite. By definition, this is equivalent to showing that $\phi_{\bos}:\Isom_S(\cC,\cC')_{\bos}\to S_{\bos}$ is finite. 
Consider the split models $\c_{\spli}:=\c\times_S S_{\bos}\to S_{\bos}$ and $\c'_{\spli}:=\c'\times_S S_{\bos}\to S_{\bos}$ of, respectively, $\c\to S$ and $\c'\to S$.
Consider the fibered product $\Isom_S(\c,\c')\times_S S_{\bos}$ with its two projections $p_1: \Isom_S(\c,\c')\times_S S_{\bos}\to  \Isom_S(\c,\c')$ and $p_2:\Isom_S(\c,\c')\times_S S_{\bos}\to S_{\bos}$. 
By construction, $\Isom_S(\c,\c')\times_S S_{\bos}$ represents the isomorphism functor from $\c_{\spli}$ to $\c'_{\spli}$ and also, by Proposition  \ref{susy_spin}, the isomorphism functor between the associated spin curves over $S_{\bos}$. Since the stack $\S_g$ of spin curves is separated, we conclude that  the second projection $p_2$ is finite. 
Observe also that the first projection $p_1$ is a closed embedding since it is the base change of the closed embedding $S_{\bos}\hookrightarrow S$. 

Now the natural inclusion map $\iota: \Isom_S(\c,\c')_{\bos} \to \Isom_S(\c,\c')$ together with $\phi_{\bos}$   determine a map $\rho=(\iota,\phi_{\bos}): \Isom_S(\c,\c')_{\bos}\to \Isom_S(\c,\c')\times_S S_{\bos}$ such that $\phi_{\bos}=p_2\circ \rho$ and $\iota=p_1\circ \rho$.
Since $\iota$ and $p_1$ are closed embeddings, also $\rho$ is a closed embedding. Hence we conclude that $\phi_{\bos}=p_2\circ \rho$ is finite because $p_2$ is finite and $\rho$ is a closed embedding.


Finally, we  have show that the kernel of $d\phi_x$ is trivial for any complex point $x\in  \Isom_S(\c,\c')$. This kernel consists of global vector fields on the susy curve $\c_{\phi(x)}\cong \c'_{\phi(x)}$ which commute with the susy structure.  But the only such global vector field is the zero vector field as proved in \cite[Prop. 2.2]{LeBrun}. 




\end{proof}

\subsection{Kuranishi family}\label{S:Kuranishi}

We introduce the notion of deformation and Kuranishi family. A pointed complex superspace $(S,0)$ is a complex superspace $S$ together with a marked point $0$.

\begin{definition}[Deformation] Let $C_L$ be a susy curve over a point, and $(S,0)$ a pointed complex superspace . A deformation of $C_L$ over $(S,0)$ is a susy curve $\c$ over $S$ and an isomorphism between $C_L$ and the fibre of $\c$ over $0$.
\end{definition}

\begin{definition}[Kuranishi family]
Let $C_L$ be a susy curve over a point. A Kuranishi family for $C_L$ is a deformation of $C_L$ over a base $(U,0)$ such that for any other deformation $\c'$ of $C_L$ over a base $(S,0)$ there exists, possibly up to shrinking $S$, a unique morphism from $(S,0)$ to $(U,0)$ such that $\c'$ is the pull-back of $\c$.
\end{definition}

For a general introduction to Kuranishi families the reader can look at \cite[Section 11.4]{ACG}. In this set up, combining the existence of Kuranishi families for spin curves with \cite[Thm. 2.8]{LeBrun} or \cite[3.4.5]{Vaintrob}, we have the following theorem

\begin{theorem}[Existence of Kuranishi family]\label{thm:exist_kur}
Let $C_L$ be a susy curve of genus $g\geq 2$ over a point. Let $\pi:(\ov C, \L)\to (V,0)$ be a Kuranishi family for the spin curve $(C,L)$, where $V$ is complex manifold of dimension $3g-3$. Denote by $(U,0)$ the split complex supermanifold of dimension $3g-3|2g-2$ associated to $(V,0)$ and the locally free sheaf  $R^1\pi_*L$. Then, $(U,0)$ is the base of a Kuranishi family $\c$ for the susy curve $C_L$. 
\end{theorem}


The following result on Kuranishi families will be crucial for the next section. 

\begin{proposition}\label{prop:iso}
Let $\pi\colon \c \to (U,0)$ be a Kuranishi family for the susy curve $C_L$ over $\C$. Consider the two susy curves $p_i^*\c\to U\times U$ where $p_i:U\times U\to U$ is the $i$-th projection, for $i=1,2$. 
Then the morphism 
$$
\psi: \Isom_{U\times U}(p_1^*\c,p_2^*\c)\xrightarrow{\phi} U\times U \xrightarrow{p_1} U
$$
is \'etale and surjective. 
\end{proposition}
\begin{proof}

Let $\Delta$ denote the diagonal of $U\times U$. Since $p_1^*\c_{|\Delta}= p_2^*\c_{|\Delta}$, the image of $\phi$ contains the diagonal $\Delta$. Since $p_1$ maps $\Delta$ isomorphically into $U$, we deduce that $\psi$ is surjective.


Let $x$ be  point of $\Isom_{U\times U}(p_1^*(\c),p_2^*(\c))$ and set $(u_1,u_1):=\phi(x)\in U\times U$. To prove that $\psi$ is \'{e}tale at $x$, it is enough to show that $\psi$ is locally an isomorphism around $x$. Since $\c\to U$ is a Kuranishi family at $u_1$ and $u_2$ (by the  openness of versality), 
we can find open neighbourhoods $u_1\in V_1\subseteq U$ and $u_2\in V_2\subseteq U$  such that there exists a unique pair of isomorphisms $f:V_1\to V_2$ and $F:\c_{|U_1}\to f^*(\c_{|U_2})$. Then the pair $(f,F)$ defines the inverse of $\psi$ from the open neighbourhood $V_1\subseteq U$ of $\psi(x)=u_1$ to an open  neighbourhood of $x$ in $\Isom_{U\times U}(p_1^*(\c),p_2^*(\c))$. 


\end{proof}

\begin{section}{The moduli superstack and superspace of susy curves}\label{sec:moduli}

We are now ready to study the moduli superstack of susy curves, using the language introduced in Section \ref{sec:superstack}. 

\begin{definition}[Moduli CFG of genus $g$ susy curves]
Let $g\geq 2$. The category fibered in groupoids over $\S$ of genus $g$ susy curves is the CFG whose objects are  
susy curves $\pi \colon \c \to S$ over some complex superspace  $S$, whose morphisms are Cartesian diagrams, and whose fibration is  the forgetful functor  $F(\pi \colon \c \to S)=S$.
\end{definition}

Proposition  \ref{susy_spin} implies that the bosonic truncation $(\sM_g)_{\bos}$ of $\sM_g$ (as in \eqref{E:funct-bos}) is the CFG  over $\S_{\ev}$  of genus $g$ spin curves, denoted by $\S_g$, whose objects are spin curves $(C,L,\phi)\to S$ over some complex space $S\in \S_{\ev}$, whose morphisms are Cartesian diagrams and whose fibration is the forgetful functor $F((C,L,\phi)\to S)=S$.
We are going to use the following well-known properties of $\S_g$.

\begin{fact}\label{F:Sg}
Let $g\geq 2$.
\begin{enumerate}[(i)]
\item \label{F:Sg1} $\S_g$ is a smooth and separated DM complex stack of dimension $3g-3$. $\S_g$ has two connected components: $\S_g^+$  parametrizes respectively even spin curves, i.e.  spin curves $(C,L,\phi)$ over $\C$ such that $h^0(L)$ is even; 
$\S_g^-$  parametrizes respectively odd spin curves, i.e. spin curves $(C,L,\phi)$ over $\C$ such that $h^0(L)$ is odd. 
\item \label{F:Sg2} $\S_g$  has a coarse moduli space $S_g$ which is a quasi-projective variety and it does have two connected components $S_g^+$ and $S_g^-$ which are the coarse moduli spaces of $\S_g^+$ and $\S_g^-$, respectively. 
\end{enumerate}
\end{fact}

In the next theorem, we show that $\sM_g$ is a complex superstack and we collect its geometric properties.

\begin{theorem}\label{T:sS-superstack}
For any $g\geq 2$, $\sM_g$ is a smooth and separated DM complex superstack of dimension $3g-3|2g-2$ such that $(\sM_g)_{\bos}=\S_g$. Moreover, $\sM_g$ has two connected components, denoted by $\sM_g^+$ and $\sM_g^-$, whose bosonic truncations are $(\sM_g)_{\bos}^+=\S_g^+$ and $(\sM_g)_{\bos}^-=\S_g^-$.
\end{theorem}
\begin{proof}

Let us prove that $\sM_g$ is a complex superstack by checking the conditions of Definition \ref{D:superstack} and  Definition \ref{def:DMstack}.

Condition \eqref{stack2} of Definition \ref{D:superstack} and condition \eqref{def:DMstack1} of Definition \ref{def:DMstack} follow from Remark \ref{scorciatoia} and Proposition \ref{prop:auto}.


Let us now show that condition  \eqref{stack1} of Definition \ref{D:superstack} holds true for $\sM_g$, i.e. that descent data for susy curves are effective.
We are going to use the notation introduced in the definition of descend data. We first pull-back everything to $T_{\bos}$ and $S_{\bos}$. Here, we are working on a spin curve, and we already know that descend data are effective, so we obtain a susy curve $\c_{\spli}$ over $S_{\bos}$. To extend it to a susy curve over $S$, we need to construct just the structure sheaf and the susy structure, because the topological space is already there. The descent data from $\c'$  give us a open cover of $\c$, and a sheaf and the susy structure on each open set. We use the descent data to glue together these sheaves and get the structure sheaf $\O_{\c}$ and the susy structure.

Let us finally show that condition \eqref{def:DMstack2} of Definition \ref{def:DMstack} holds true for $\sM_g$, i.e. that $\sM_g$ admits an \'etale atlas. 
 To this end we use Kuranishi families, whose existence has been proved in Theorem \ref{thm:exist_kur}. Since $\S_g$ can be covered by a finite disjoint union of Kuranishi families of spin curves, and $(\sM_g)_{\bos}=\S_g$, taking a suitable disjoint union $X$ of a finite number Kuranishi families of susy curve, we obtain a complex superspace $X$ with a map towards $\sM_g$ such that $X_{\bos}$ is an atlas for $\S_g$. Now, the fact that the map $X\to \sM_g$, which is  surjective because the map  $X_{\bos}\to \S_g$ is surjective, is  \'{e}tale follows  from Proposition \ref{prop:iso} as in the proof of \cite[Thm. XII.8.3]{ACG}.

Since the atlas $X$ is smooth of dimension $3g-3|2g-2$ by Theorem \ref{thm:exist_kur}, we get that  $\sM_g$ is smooth of dimension $3g-3|2g-2$. Moreover, since $(\sM_g)_{\bos}=\S_g$ is separated and it has two connected components $\S_g^+$
and $\S_g^-$ by Fact \ref{F:Sg}, we deduce that  $\sM_g$ is separated and it has two connected components $\sM_g^+$ and $\sM_g^-$ whose bosonic truncations are $\S_g^+$ and $\S_g^-$, respectively. 
\end{proof}

Recall that the moduli  stack $\S_g$ of spin curves is a $\mu_2$-gerbe. This means that for every object  $\c\to S$, we have an injection of $\mu_2$ in the group $\Aut_S(\c)$ of automorphisms of $\c$ over $S$,
and this is compatible with base change. In this set up, we can construct the rigidification, which is a DM complex stack denoted by $\S_g\fatslash\mu_2$, and the  projection $\pi\colon \S_2\to \S_g\fatslash  \mu_2$ is a $\mu_2$-gerbe
(see \cite{Rom}).

In the case of the moduli superstack $\sM_g$ of susy curves, for every object $\c\to S$, the canonical involution $\Gamma_{\c}$ gives an embedding of $\mu_2$ in the group $\Aut(\c)$ of automorphisms of $\c$, and this is compatible with base change. However, the forgetful functor maps $\Gamma_{\c}$ to $\Gamma_S$, which, in general, is not equal to $Id_S$; hence, $\Gamma_{\c}$ does not belong in general to $\Aut_S(\c)$, so that $\sM_g$ is not a $\mu_2$ gerbe. 
The existence of the automorphism $\Gamma_{\c}\in \Aut(\c)$ can be rephrased by saying that, given a family of susy curves $\c\to S$, which is equivalent to giving a map from $S$ to $\sM_g$, we have an embedding $\gamma=Id_S\times \Gamma_S$ of $S$ into $S \times_{\sM_g} S=\Isom_{S\times S}(p_1^*\c,p_2^*\c)$, see also \cite[Chap. XII, Equ. (8.5)]{ACG}. This ``gerbe-like" structure, which, in some broad sense, could be typical of all moduli superstacks of super objects, implies the following proposition. 
\begin{proposition}\label{prop:fact}
Let $M$ be a complex superspace, any map $\phi\colon \sM_g\to M$ factors trough the quotient $\sM_g/\Gamma$.
\end{proposition}
\begin{proof}
Recall that a supergroupoid presentation of $M$ is $M_{\bullet}=M\rightrightarrows M$, where both arrows are the identity. Since the pseudofunctor $\F_{\loc}$ of \eqref{E:loc-mor} is an isomorphism of bicategories, we can pick a supergroupoid presentation $X_{\bullet}=X_1\rightrightarrows X_0$ of $\sM_g$ such that $\phi$ is induced by a map of groupoids $f=(f_0,f_1)\colon X_{\bullet}\to M_{\bullet}$, i.e. $\phi=\F_{\loc}(f)$. Since $\sM_g/\Gamma=[X_1/\Gamma\rightrightarrows X_0/\Gamma]$ by definition (see Theorem \ref{T:DMtrunc}\eqref{T:DMtrunc2}), it is enough to show that $f$ factors through the quotient $\pi=(\pi_0,\pi_1):X_{\bullet}\to X_{\bullet}/\Gamma$.

To this aim, consider the map  $\gamma=Id_{X_0}\times \Gamma_{X_0}\colon X_0\to X_1$ discussed above and the diagonal embedding $\Delta: X_0\to X_1=X_0\times_{\sM_g} X_0$. Since $\gamma$ and $\Delta$ are both sections of the source map 
$s:X_1\to X_0$ and both the arrows in the presentation of $M$ are the identity, we have that 
$$f_1\circ \Delta=f_0\circ s\circ \Delta = f_0\circ s\circ \gamma=f_1\circ \gamma.$$ 
Similarly, using the above equality and the relations $t\circ \Delta=Id_{X_0}$ and $t\circ \gamma=\Gamma_{X_0}$ where $t:X_1\to X_0$ is the target map, we get that 
$$f_0=f_0\circ t\circ \Delta= f_1\circ \Delta=f_1\circ \gamma=f_0\circ t \circ \gamma=f_0\circ \Gamma_{X_0}.$$ 
This shows that $f_0$ factors through the quotient $\pi_0:X_0\to X_0/\Gamma$. 
Moreover, using that $M_{\bos}\hookrightarrow M$ is universal with respect to morphisms from  spaces into $M$, we deduce that $f_0$ has to factor through the inclusion $M_{\bos}\hookrightarrow M$. Since $f_1=f_0\circ s$ (because the target morphism of $M_{\bullet}$ is the identity), this implies that also $f_1$ has to factor via the inclusion $M_{\bos}\hookrightarrow M$. Since the bosonic quotient is left adjoint to the inclusion of spaces into superspaces, we infer that $f_1$ has to factor via the quotient $\pi_1:X_1\to X_1/\Gamma$. 
 We conclude that $f$ factors trough $\pi$ as required.

\end{proof}

This result has the following key corollary. The idea underlying it is that, locally around a  point $[C]$, the coarse moduli space of curves looks like a Kuranishi family for $C$ modulo $\Aut(C)$, cf \cite[Section XII.2]{ACG}; when $C$ is a susy curve, $\Aut(C)$ contains $\Gamma$, hence the coarse moduli space is an ordinary non-reduced space.

\begin{corollary}\label{cor:coarse}
The coarse superspace $\mathbb{M}_g$ of $\sM_g$ does exist and it is even, or in other words it is an ordinary complex space. Moreover, $\mathbb{M}_g$ is also the coarse space of the bosonic quotient $\sM_g/\Gamma$.

The complex space $\mathbb{M}_g$ is not reduced; its reduction $(\mathbb{M}_g)_{\red}$ is isomorphic to the coarse moduli space $S_g$ of genus $g$ spin curves. In particular, $\mathbb{M}_g$ is separated, and it has two connected components.

\end{corollary}
\begin{proof}
Because of Proposition \ref{prop:fact}, the coarse superspace of $\sM_g$, if it exists, is equal to the coarse space of $\sM_g/\Gamma$; moreover, the map from $\sM_g$ to its coarse superspace factors though the quotient $\sM_g\to \sM_g/\Gamma$. The stack $\sM_g/\Gamma$ is an ordinary separated DM complex stack, hence its coarse space exists by \cite{KM}. We denote this space by $\mathbb{M}_g$.

We now focus on the second part of the corollary. The first assertions come from the fact that $\mathbb{M}_g$ is the coarse space of $\sM_g/\Gamma$, an atlas for $\sM_g/\Gamma$ is $X/\Gamma$, where $X$ is a finite union of Kuranishi families (see the proof of Theorem \ref{T:sS-superstack}), and $X/\Gamma$ is not reduced since $X$ has dimension $3g-3|2g-2$ and $2g-2\geq 2$. To describe $(\mathbb{M}_g)_{\red}$,  recall that $(\sM_g/\Gamma)_{\red}\cong (\sM_g)_{\bos}\cong \S_g$, hence $(\mathbb{M}_g)_{\red}$ is the coarse space of $\S_g$, which is nothing but $S_g$. A space $Z$ is separated if and only if $Z_{\red}$ is separated, cf. \cite[I.5.5.1 (vi)]{EGA}, hence $\mathbb{M}_g$ is separated. By the same token, since $S_g$ has two connected components, the same is true for $\mathbb{M}_g$.
\end{proof}

We do not know if $\mathbb{M}_g$ is quasi-projective; the fact that the underlying reduced space is quasi-projective does not suffice (see \cite[II.5.3.5]{EGA}). We also do not know if $\mathbb{M}_g$ can be interpreted as a moduli space of some more classical object.

\begin{remark}
There is an analogy between the situation described in Corollary \ref{cor:coarse} and the way the hyperelliptic locus sits inside the moduli stack of curves; hoping to help the reader, we briefly describe it. One can look at the first and second infinitesimal neighborhoods $H_1$ and $H_2$ of the hyperelliptic locus  inside the moduli stack of genus $g\geq3$ curves. These spaces are non-reduced DM complex  stacks. The hyperelliptic involution acts on $H_1$ and $H_2$; more specifically, it is the identity on the underlying reduced stacks, but acts non-trivially in the normal direction. The associated coarse moduli spaces of $H_1$ and $H_2$ are equal to the coarse moduli spaces of their respective quotients by the hyperelliptic involution. In particular, the coarse moduli space of $H_1$ is the coarse moduli space of hyperelliptic curves, while the coarse moduli space of $H_2$ is a thickening of the coarse moduli space of hyperelliptic curves. One might think at $H_2$ with the action of the hyperelliptic involution as a sort of toy model for $\sM_g$ with the action of the canonical involution $\Gamma$.
\end{remark}

\end{section}

\begin{section}{Periods of susy curves}\label{sec:period}
\begin{subsection}{Global period map}\label{global_per}
In this section we define periods of super symmetric curves and prove some of their properties. We are going to take an Hodge-theoretic approach; the reader can found in \cite[Section 1.2]{Carlson} a nice introduction to period matrices and period map in the classical setting.

Let $\pi \colon \c \to S$ be a susy curve, and $\pi\colon \c_{\spli}=C_L\to S_{\bos}$ its split model. The line bundle $\Ber(\c)$ is the Berenzinian of the relative cotangent bundle. The following lemma is well-known, see for instance \cite[Section 8]{Witten2} or \cite[Criterion 2]{Giulio}.

\begin{lemma}\label{lem:rank}
The line bundle $\pi_*\Ber(\c)$ is locally free on $S$ when $\pi_*L=0$ on $S_{\bos}$; in this case, its rank is $g|0$.
\end{lemma}
%

From now on, we consider susy curves $\pi:\c\to S$ such that $\pi_*L=0$. As explained in \cite[Section 8.1]{Witten2} or \cite[Section 2.4]{Ber}, we can integrate sections of the Berenzinian line bundle along closed topological $1$-cycle in $C$; moreover, the value of the integral depends just on the homology class of the cycle. This means that to any global section of the Berenzinian line bundle we can associate a cohomology class in $H^1(C,\C)$.

\medskip


The global sections of $\Ber(\c)$, because of Lemma \ref{lem:rank}, are even; in other words, they are invariant under the canonical automorphism $\Gamma$. Roughly speaking, this means that the integral of a global section of the Berenzinian line bundle along a fixed closed 1-cycle is an even function on the base $S$. Let us formalize this idea. Assume that $\c$, as a topological space, is homeomorphic to $C\times S_{\bos}$, where $C$ is a fibre. Integration gives a morphism
$$ \int \colon \pi_*\Ber(C) \to H^1(C,\Z) \otimes_{\Z} (\O_S)^{\Gamma} $$
where $H^1(C,\Z)$ is the singular co-homology of the underlining topological space. In a more general set up, we should write $R^1f_*\mathbb{Z}$ rather than $H^1(C,\mathbb{Z})$.

\medskip

Riemann bilinear relations hold (see for example \cite[Section 2.8]{Ber} and \cite[Section 8.2]{Witten2}), so the image is isotropic for the intersection pairing on $H^1(C,\Z)$. Combining this with Lemma \ref{lem:rank}, we learn that the image of the integration map is an isotropic $g$-dimensional subspace of the $2g$-dimensional $\O_S^{\Gamma}$-module $H^1(C,\Z) \otimes_{\Z} (\O_S)^{\Gamma}$. In a more fancy language, this means that we have a morphism
$$
P\colon S/\Gamma\to \A_g
$$
where $\A_g$ is the classical moduli stack of principally polarized abelian varieties of dimension $g$. We land in $\A_g$ because $(\O_S)^{\Gamma}$ is a sheaf of even commutative ring on $S_{\bos}$. Similarly, since the values of the integral is invariant under $\Gamma$, the morphism $P$ factors trough the quotient $S/\Gamma$. 

It is well-known that the image of a split curve is just the Jacobian of the reduced curve; let us write out the computation for the reader convenience.
\begin{lemma}\label{lem:class_period}
The period matrix of a split curve $C_L$ with $h^0(L)=0$ equals the period matrix of $C$.
\end{lemma}
\begin{proof}
Fix a symplectic basis $A_i$, $B_i$ for the homology of $C$, let $\o_1,\dots ,\o_g$ be the corresponding basis of the holomorphic differentials and $\tau$ the period matrix. Let $(z_{\alpha},\theta_{\alpha})$ be a conformal atlas for $C_L$. Locally we write
\[\o_i=f_i(z_{\alpha})dz_{\alpha}\,.\]
Locally, a basis of $H^0(C_L,\Ber(C_L))$ is given by
\[\hat{\o}_i=f_i(z_{\alpha})\theta_{\alpha}[dz_{\alpha}\mid d\theta_{\alpha}]\]
since  can check directly that $f_i(z_{\alpha})\theta_{\alpha}$ transforms as a section of the Berenzinian. We are using that $h^0(L)=0$ to assume that $\hat{\o}_i$ spans all $H^0(C_L,\Ber(C_L))$. Integrating $\hat{\o}_i$ along the cycles $B_j$ one gets again the period matrix $\tau$.
\end{proof}

The morphism $P$ we defined above is functorial; so we have a period map

$$
\wt P\colon \sM_g^+\dashrightarrow \A_g,
$$
where $\sM_g^+$ is the connected component of $\sM_g$ such that $(\sM_g^+)_{\bos}=\S_g^+$ is moduli stack of even spin curves. The locus where $\wt P$ is not defined is theta null divisor, which consists of all even spin curves $(C,L, \phi)$ such that $h^0(L)>0$. Lemma \ref{lem:class_period} implies that $\wt P$ is an extension of the classical period map. In other words, $(\wt P)_{\bos}$ on $(\sM_g^+)_{\bos}$ maps a spin curve $(C,L, \phi)$ to the Jacobian $J(C)$ of $C$. 


We can define a slightly more refined version of the period map; namely, we can lift it to a morphism
$$
P\colon \sM_g^+\dashrightarrow \N_g,
$$
where $\N_g$ is the moduli stack of abelian varieties endowed with a symmetric theta divisor. To do this, we first observe that $(\wt P)_{\bos}$ admits a natural lifting to $\N_g$ by mapping a spin curve $(C,L, \phi)$ to the Jacobian of $C$ endowed with the symmetric theta divisor associated to the spin structure $L$, i.e. the translation via $L^{-1}$ of the natural theta divisor $\Theta_C\in {\rm Pic}^g(C)$. The morphism from $\N_g$ to $\A_g$ is \'{e}tale, so the lift of $(\wt P)_{\bos}$ gives a lift of $\wt P$. The classical Torelli theorem implies that this new morphism $P$ is injective on complex points. However, the morphism $P$ is far from being an embedding since it factors trough the quotient by the canonical automorphism, so that we can write
$$
P/\Gamma\colon \sM_g^+/\Gamma \dashrightarrow \N_g.
$$
The moduli stack $\N_g$ admits a coarse space $N_g$, so by the universal properties of the coarse moduli spaces we have a third version of the period map at the level of coarse moduli spaces
$$
\ov P\colon \mathbb{M}_g^+ \dashrightarrow N_g.
$$

\end{subsection}

\begin{subsection}{The classical formula}\label{sec:class_formula}
There is a classical formula for the super period map: it was first discovered by D'Hoker and Phong in \cite{DP}; their argument has been improved and expanded in \cite[Section 8.3]{Witten2}. Both the proof and the statement are quite analytic; let us briefly explain their formula.

We fix a complex point $[C_L]$ of $\sM_g^+$ such that $h^0(L)=0$; we work locally on a Kuranishi family for $C_L$.  On the base we have $2g-2$ odd coordinates, let us call them $\eta_1,\dots , \eta_{2g-2}$. The local co-ordinates come from a basis of the co-tangent space; so the $\eta_i$ are a basis of $H^0(C,K_C\otimes L)$. After fixing a symplectic basis of the homology of $C$, we can now consider the period matrix $\tau=(\tau_{i,j})$  as a function on the base. We want to study the Taylor expansion of $\tau$ in the odd variables around the point $0:=[C_L]$. Being the period matrix an even function, this expansion will look like
$$
\tau_{i,j}=\tau_{i,j}(0)+\sum_{a,b}\sigma_{a,b}^{(i,j)}\eta_a\eta_b 
$$
modulo functions of order at least $3$ in $\eta_i$. We omit the dependance on the even moduli. The quantity $\sigma_{a,b}^{(i,j)}$ lives in tangent space to $\N_g$ at $\tau(0)$. In \cite[Section 8.3 Formula 8.38]{Witten2}, the following formula is proven
\begin{equation}\label{formula}
\sigma_{a,b}^{(i,j)}=\int_{C\times C}p_1^*(f_a\omega_i) \wedge S_L  \wedge p_2^*(\omega_jf_b).
\end{equation}
In this formula, the $f_i$ are a basis of $H^1(C,L^{-1})$ Serre-dual to the $\eta_i$, so they are $L^{-1}$-valued $(0,1)$ forms on $C$. The $\omega_i$ are a normalized basis of $H^0(C,K)$. The kernel $S_L$ is the Szeg\H{o} kernel associated to the spin structure $L$; we are going to describe it in Section \ref{sec:Szego}. The $p_i$ are the projections on the two factors of $C\times C$.


\end{subsection}

\begin{subsection}{Infinitesimal period map}\label{sec:inf_per}

We want to study the differential of the period map 
$$
P\colon \sM_g^+\dashrightarrow \N_g
$$
from an algebraic point of view. Being the period map  $\Gamma$-invariant, the odd tangent space of $\sM_g^+$ is in the kernel of the differential of $dP$. It is thus more interesting to look at the differential of the period map for the quotient by $\Gamma$
$$
P/\Gamma\colon \sM_g^+/\Gamma \dashrightarrow \N_g.
$$
Fix a complex point $[C_L]$ in $\sM_g^+$ such that $h^0(C,L)=0$. Our goal is to describe the differential at the point $C_L$
$$
 d(P/\Gamma)_{C_L}\colon T_{[C_L]}(\sM_g/\Gamma) \to T_{P(C_L)}\N_g,
$$ 
which is called the \emph{infinitesimal period map} at $C_L$.

To start with, let us describe explicitly the domain and the codomain of $d(P/\Gamma)_{C_L}$ and some background about the Szeg\H{o} kernel $S_L$ associated to $L$. 

%
%

\begin{subsubsection}{First infinitesimal neighborhood of $C_L$ in $\sM_g/\Gamma$}
Thanks to the description of the Kuranishi family for $C_L$ given in Theorem \ref{thm:exist_kur}, we know that the tangent space of $\sM_g/\Gamma$ at $[C_L]$ splits non-canonically into  the sum of $\bigwedge^2H^1(C,\sT)$ and $H^1(C,T_C)$. To avoid this non-canonical splitting, we give the following more intrinsic description. 

Consider  the following exact sequence associated to the diagonal $\Delta$ inside the surface $C\times C$:
\begin{equation}\label{preAg}
0\to \O \to \O(\Delta)\to \O_{\Delta}(\Delta)=T_C\to 0,
\end{equation}
where we have used the well-known fact that the normal sheaf of $\Delta$ inside $C\times C$ is isomorphic to the tangent sheaf $T_C$ once we identify $\Delta$ with the curve $C$. 
By tensoring \eqref{preAg} with $\sT\boxtimes \sT(-\Delta)$ we get the exact sequence 
\begin{equation}\label{presM}
0\to \sT\boxtimes \sT(-\Delta) \to \sT\boxtimes \sT\to \sT\boxtimes \sT|_{\Delta}=T_C\to 0.
\end{equation} 
Taking cohomology we obtain the exact sequence of $\C$-vector spaces
\begin{equation*}
0\to  H^1(C,T_C)  \to H^2(C\times C, \sT\boxtimes \sT(-\Delta)) \to H^1(C,\sT)^{\otimes 2}\to 0.
\end{equation*}
The invariant part of the above sequence with respect to the canonical involution $\iota$ on $C\times C$ switching the two factors (using the sign conventions of \cite[Sec. 3.1]{DW2}) is
\begin{equation}\label{sM}
0\to  H^1(C,T_C)  \to H^2(C\times C, \sT\boxtimes \sT(-\Delta))^+=T_{C_L}(\sM_g/\Gamma) \to \bigwedge^2H^1(C,\sT)\to 0.
\end{equation}
This sequence coincides with the sequence $(3.13)^+$ of \cite{DW2}, and it is the dual of the sequence $(3.11)^+$ of \cite{DW2} that represents the first obstruction to the splitness of the superstack $\sM_g$ of susy curves over the stack of spin curves $\S_g$.


\end{subsubsection}

\begin{subsubsection}{First infinitesimal neighborhood of $P(C_L)$ in $\N_g$}

Since the map $\N_g\to \A_g$ is \'etale and the image of $P(C_L)$ in $\A_g$ is equal to the Jacobian of $C$ by Lemma \ref{lem:class_period}, the tangent space $T_{P(C_L)}\N_g$ is equal to $T_{J(C)}\A_g$, which is well-known to be equal to 
$\Sym^2 H^1(C,\O)$.
We want now to describe  $\Sym^2 H^1(C,\O)$ as a middle term of a canonical short exact sequence which is suitable for the study of periods in the super setting. We look at the following part of the long exact sequence associated to \eqref{preAg}
\begin{equation}\label{map-c}
H^1(C,T_C)\xrightarrow{c} H^2(C\times C,\mathcal{O})\to H^2(C\times C, \mathcal{O}(\Delta))\to 0 
\end{equation}

\begin{proposition}\label{P:Sym2}
We have 
$$ H^1(C,T_C)^+=H^1(C,T_C) \quad \textrm{and} \quad H^2(C\times C,\mathcal{O}_{C\times C})^+= \Sym^2 H^1(C,\mathcal{O}_C)\,,$$
and the even part of the map  map $c$  from \eqref{map-c}  is Serre dual to the multiplication map
$$
m\colon \Sym^2H^0(C,K_C)\to H^0(C,2K_C).
$$
\end{proposition}
\begin{proof}
We first compute the Serre dual of the sequence  \eqref{preAg}. Applying the functor ${\mathcal Hom }(-,\O)$ to \eqref{preAg}, we get the following short exact sequence of sheaves on $C\times C$
\begin{equation}\label{dual}
0 \to \O(-\Delta) \to \O \xrightarrow{\delta_1} {\mathcal Ext}^1(\O_{\Delta}(\Delta),\O) \to 0,
\end{equation}
where $\delta_1$ is the first coboundary map. In the sequence \eqref{dual}, we have ${\mathcal Ext }^1(\O_{\Delta}(\Delta),\O) =\O_{\Delta}$ and the coboundary map $\delta_1$ is the restriction to the diagonal. This follows from the fact that, if  the genus of the curve $C$ is at least $2$, there exists a unique up to a scalar morphism from $\O(-\Delta)$ to $\O$. Tensoring the exact sequence \eqref{dual} with the canonical bundle $K_{C\times C}=K_C\boxtimes K_C$, we thus obtain the exact sequence 
\begin{equation}\label{Serre-dual}
0\to  K_{C\times C}(-\Delta) \to K_{C\times C}  \xrightarrow{\delta_1}   K_{C\times C}|_{\Delta}=2K_C \to 0,
\end{equation}
which by definition is the Serre dual of \eqref{preAg}. The map $\delta_1$ is just the restriction to the diagonal, so $H^0(\delta_1)$ is the multiplication.

 As explained in \cite[Section 3.3]{DW2}, the sheaf $2K_C$ supported on the diagonal $\Delta$ is entirely even (in the notation of  \cite[Section 3.3]{DW2}, we are taking $a=1$ and $c=0$). Arguing as in \cite{DW2}, we deduce that that $H^0(C,2K_C)^+=H^0(C,2K_C)$, and  $H^0(C\times C, K_{C\times C})^+=\Sym^2H^0(C,K_C)$. We obtain the proposition by Serre duality.

\end{proof}

Recall also that the kernel of the multiplication mao  $m$ is the vector space $I_2(C)$ of  quadrics containing the canonical model of $C$ and that $m$ is surjective if and only if $C$ is not hyperelliptic by Noether's Theorem. 





\end{subsubsection}

\begin{subsubsection}{The Szeg\H{o} kernel $S_L$}\label{sec:Szego}

In this section  we recall the definition of the Szeg\H{o} kernel $S_L$ associated to a theta characteristic $L$ such that  $h^0(\sK)=0$. We adopt the algebraic point of view of \cite{BZB} rather than the classical analytic approach. 

Tensoring the exact sequence \eqref{preAg} by $\sK\boxtimes \sK$, we obtain the exact sequence
$$
0\to \sK\boxtimes \sK \to \sK\boxtimes \sK (\Delta) \to \O_{\Delta}\to 0.
$$
Since $h^0(\sK)=h^1(\sK)=0$, we obtain an isomorphism
$$
H^0(C\times C,\sK\boxtimes \sK (\Delta))\stackrel{\cong}{\longrightarrow} H^0(\Delta, \O_{\Delta})\cong \C.
$$
The Szeg\H{o} kernel $S_L$ is the preimage of $1\in H^0(\Delta, \O_{\Delta})\cong \C$ under the above isomorphism. 

Notice that, since $\sK=K_C\otimes \sK^{-1}$, the above isomorphism can also be identified with the residue map along the diagonal  
$$H^0(C\times C,\sK\boxtimes \sK (\Delta)) \xrightarrow{\mathrm{Res}} H^0(C,\End(\sK)).$$ 
Therefore, the Szeg\H{o} kernel $S_L$ is the preimage of the identity under the residue map $\mathrm{Res}$.

The  Szeg\H{o} kernel is symmetric with respect ot the involution on $C\times C$, as explained for instance in \cite[Remark 5.1.4]{BZB}.

\end{subsubsection}
\begin{subsubsection}{The infinitesimal period map}

We are now in position to describe the infinitesimal period map
$$ 
d(P/\Gamma)_{C_L}\colon T_{[C_L]}(\sM_g^+/\Gamma) \to T_{P(C_L)}\N_g, 
$$
under the assumption $h^0(\sK)=0$. 
Recall that $T_{[C_L]}(\sM_g^+/\Gamma)=H^2( \sT\boxtimes \sT(-\Delta))^+$ sits in the exact sequence \eqref{sM} while $T_{P(C_L)}\N_g=\Sym^2H^1(C,\O)$ sits in the exact sequence 
\begin{equation}\label{Ag}
H^1(C,T_C)  \xrightarrow{c=m^{\vee}} \Sym^2H^1(C,\O) =T_{P(C_L)}\N_g \to I_2^{\vee}  \to 0,
\end{equation}
by Proposition \ref{P:Sym2}.


The infinitesimal period map $d(P/\Gamma)_{C_L}$ is a morphism
$$ 
d(P/\Gamma)_{C_L}\colon H^2( \sT\boxtimes \sT(-\Delta))^+\to \Sym^2H^1(C,\O)
$$
which is the identity on $H^1(C,T_C)$ since it extends the differential of the classical period map $P_{\bos}:\S_g\to \N_g$.

\begin{lemma}\label{unique}
Assume that $h^0(\sK)=0$. 
The multiplication by the Szeg\H{o} kernel $S_L$ is the unique morphism form the exact sequence of sheaves (\ref{presM}) on $C\times C$ 
\begin{displaymath}
0\to \sT\boxtimes \sT(-\Delta) \to \sT\boxtimes \sT\to \sT\boxtimes \sT|_{\Delta}=T_C\to 0
\end{displaymath} 
to the exact sequence (\ref{preAg}) on $C\times C$
\begin{displaymath}
0\to \O \to \O(\Delta)\to \O_{\Delta}(\Delta)=T_C\to 0
\end{displaymath}
which extends the identity on $T_C$. 
\end{lemma}
\begin{proof}
The multiplication by $S_L$ restricts to the identity on $T_C$ because of the definition of $S_L$ (see the end of  \ref{sec:Szego}). The claim now follows from the fact  $h^0(C\times C,\sK\boxtimes \sK (\Delta))=1$, which was proved in \ref{sec:Szego}.
\end{proof}

Our description of the infinitesimal period map is contained in the following

\begin{theorem}\label{diff_per}
Fix a complex point $[C_L]$ in $\sM_g^+$ such that $h^0(C,L)=0$. The infinitesimal period map
$$ 
d(P/\Gamma)_{C_L}\colon T_{[C_L]}(\sM_g^+/\Gamma)=H^2( C\times C, \sT\boxtimes \sT(-\Delta))^+\to T_{P(C_L)}\N_g =\Sym^2H^1(C,\O)
$$
is the map
$$H^2( C\times C, \sT\boxtimes \sT(-\Delta))^+\to \Sym^2 H^1(C,\O)$$
given by the even part of the $H^2$ of the morphism of sheaves on $C\times C$
$$
\sT\boxtimes \sT(-\Delta) \hookrightarrow \O
$$
defined by the multiplication with the Szeg\H{o} kernel $S_L$ associated to $L$.

In particular, the differential $d(P/\Gamma)_{C_L}$ induces a morphism from the exact sequence \eqref{sM} to the  exact sequence  \eqref{Ag}, which restricts to the identity on $H^1(C,T_C)$.
\end{theorem}

If we were able to prove that, \emph{a priori}, the infinitesimal period map has to be the push forward of a morphism of sheaves on $C\times C$, then Theorem \ref{diff_per} would follow form Lemma \ref{unique}. Unfortunately, we are not able to complete this argument; we claim that Theorem \ref{diff_per} is correct because it matches up with the formula described in Section \ref{sec:class_formula}.


 Now, let us draw a consequence from our description.
 
\begin{theorem}\label{torelli}
Fix a complex point $[C_L]$ in $\sM_g^+$ such that $h^0(C,L)=0$. The infinitesimal period map 
$$ 
d(P/\Gamma)_{C_L}\colon T_{[C_L]}(\sM_g^+/\Gamma) \to T_{P(C_L)}\N_g 
$$
is surjective.
\end{theorem}
\begin{proof}
We use the description of the differential of the period map given in Theorem \ref{diff_per}. Let $D$ be the zero divisor of the Szeg\H{o} kernel $S_L$ and consider the  exact sequence of sheaves on $C\times C$ 
$$
0\to \O \xrightarrow{\cdot S_L}  \sK\boxtimes \sK (\Delta) \to \O(D)|_D \to 0,
$$
where the first morphism is induced by the multiplication by $S_L$.
Tensoring by $\sT\boxtimes \sT(-\Delta)$ we obtain a sequence 
$$
 0 \to \sT\boxtimes \sT(-\Delta)  \xrightarrow{\cdot S_L}   \O\to \F:= \sT\boxtimes \sT(-\Delta)\otimes \O(D)|_D  \to 0.
 $$
 
Remark that $\F$ is still supported on $D$. The period map is obtained by taking the even part of the cohomology, namely we have the exact sequence
$$
 H^2( C\times C, \sT\boxtimes \sT(-\Delta))^+\xrightarrow{d(P/\Gamma)_{C_L}}  \Sym^2H^1(C,\mathcal{O}) \to H^2(C\times C,\F)^+
$$
Since $\F$ is supported on the (one-dimensional) divisor $D$, we must have that $H^2(C\times C,\F)=0$, which then gives the surjectivity of $d(P/\Gamma)$ as claimed. 
\end{proof}


\begin{remark}\label{R:not-inj}
The infinitesimal period map described above can not be injective for dimensional reasons since 
$$
 \dim T_{[C_L]}(\sM_g^+/\Gamma)=3g-3+\binom{2g-2}{2}> \dim T_{P(C_L)}\N_g=\binom{g+1}{2}.$$
\end{remark}

\begin{remark}[Super Schottky problem] Composing with the projection from $\N_g$ to the moduli stack of principally polarized abelian varieties $\A_g$, we can look at the image of $\sM_g^+$ in $\A_g$ via the period map. This is a schematic thickening of the image of the classical moduli space $\M_g$, and it make sense to ask for its description. This should be the super version of the Schottky problem. In genus $4$, the image of $\sM_4^+$ is, up to embedded points, cut out by a power of the classical Schottky form. In view of Theorem \ref{torelli}, one needs to take at least the second power.
\end{remark}

\end{subsubsection}

\end{subsection}
\end{section}

\end{document}